\documentclass[12pt]{amsart}
\usepackage{amsmath, amssymb, amsthm,  epsfig, pdfsync,fullpage}
\theoremstyle{plain}
\newtheorem{thrm}{Theorem}[section]
\newtheorem*{thrm*}{Theorem}
\newtheorem{lemma}[thrm]{Lemma}
\newtheorem{prop}[thrm]{Proposition}
\newtheorem{cor}[thrm]{Corollary}
\theoremstyle{definition}
\newtheorem{dfn}[thrm]{Definition}
\theoremstyle{remark}
\newtheorem{rmrk}[thrm]{Remark}

\numberwithin{equation}{section}

\DeclareMathOperator{\tr}{tr}

\newcommand{\R}{\mathbb R}
\newcommand{\N}{\mathbb N}

\newcommand{\al}{\alpha}

\newcommand{\p}{\partial}  

\newcommand{\ddu}{\det du}

\newcommand{\Om}{\Omega}

\newcommand{\e}{\epsilon}
\newcommand{\vn}{\vec{n}}

\newcommand{\Fp}{\mathcal F_p}
\newcommand{\Fpo}{\mathcal F_p (u,\Omega)}

\DeclareMathOperator{\cof}{cof}
\newcommand{\cofu} {\cof du}

\newcommand{\tu}{\tilde u}
\newcommand{\tL}{\tilde{L}}
\newcommand{\fbeta}{\mathfrak b}
\newcommand{\K}{\mathbb K}
\newcommand{\pa}{\p_{par}}

\begin{document}

\thanks{\today}

\title{An Aronsson Type Approach to Extremal Quasiconformal Mappings}

\author{Luca Capogna}\address{Department of Mathematical Sciences,
University of Arkansas, Fayetteville, AR 72701}\email{lcapogna@uark.edu} 
\author{Andrew Raich}\address{Department of Mathematical Sciences,
University of Arkansas, Fayetteville, AR 72701}\email{araich@uark.edu} 

\dedicatory{\it In memory of Juha Heinonen}

\thanks{The first author is partially supported by NSF grant DMS-0800522 and
second author is partially supported by NSF grant DMS-0855822}



\begin{abstract} We study $C^2$ extremal quasiconformal mappings in space
and establish  necessary and sufficient  conditions for a `localized' form of  extremality in the spirit of the work of G. Aronsson on absolutely minimizing Lipschitz extensions. We also prove short time existence for smooth solutions of a gradient flow of QC diffeomorphisms associated to the extremal problem.
 \end{abstract}
\maketitle

%
%
\section{Introduction}

 A quasiconformal (qc) mapping is a homeomorphism, $u:\Om\subset \R^n\to \R^n$ whose components are in the Sobolev space $W^{1,n}_{loc}$ and such that 
 there exists a constant $K\ge \sqrt{n}$ for which $|du|^n \le K \det du$ a.e. in $\Om$. Here
 we denote  $|A|^2=\sum_{i,j=1}^n a_{ij}^2$ to be the Hilbert-Schmidt norm of a 
 matrix and $du$ the Jacobian matrix of $u= (u^1,...,u^n)$ with entries $du_{ij}=\p_j u^i$. At a point of differentiability  $du(x)$ maps spheres
 into ellipsoids and the smallest possible $K$ in the inequality above, roughly provides a bound for the ratio of the largest and smallest axes of such ellipsoids.
 In this sense qc mappings distort the geometry of the ambient space in a controlled fashion. Quoting F.\ Gehring \cite{Gehring:2005}, 
 qc mappings ``\emph{constitute
 a closed class of mappings interpolating between homeomorphisms and diffeomorphisms for which many results of geometric topology hold regardless of dimension}."

 Quasiconformality can be measured in terms of several {\em dilation functions}. Here we will focus on the {\em trace dilation}
 \begin{equation}\label{eqtn: analytic def}
 \K(u,\Om)= \|\K_u(x) {\|}_{L^{\infty}(\Om)} \text{ with }\K_u(x)= \frac{ |du (x)|}{(\det du (x))^{\frac{1}{n}}}.
 \end{equation}
 Other dilation functionals used in the literature are the outer, inner and linear dilation (see \cite{vaisala} for more details) 
 as well as mean dilations for mappings with finite distortion (see \cite{aimo}).

 There are a variety of 
extremal mapping problems in the theory of qc mappings, in fact qc mappings were introduced in just such a context  in \cite{Grotzsch:1928}. 
Extremal problems usually involve two domains $\Omega, \Omega'\subset \R^n$, (or two Riemann surfaces) for which there exists
a quasiconformal mapping $f:\Omega\to\Omega'$, and ask for a quasiconformal map
$u:\Om\to\Om'$ which minimizes a  dilation function
in a given class of competitors. Such competitors are usually other quasiconformal mappings with same boundary data as $f$ on a portion (or all) of $\p\Om$
or in the same homotopy class as the given map $f$. Existence and uniqueness of extremals depend strongly on the dilation function used. Typically,  
existence follows from compactness and lower-semicontinuity arguments
applied to a particular dilation function, and uniqueness does not hold unless the class of competitors is suitably restricted 
(for instance to Teichm\"uller mappings\footnote{Roughly speaking, a planar qc mapping $f$ is Teichm\"uller
if there exist local conformal transformations $\phi, \psi$ such that $\phi \circ f\circ \psi^{-1}$ 
is affine and $\phi$ and $\psi$ give rise to well defined quadratic differentials.}).

Quasiconformal extremal problem arose first in the work of Gr\"otzsch in the late 1920's
and were later studied in the two dimensional case both for open sets and for Riemann surfaces, see for instance \cite{Teichmuller}, \cite{Ahlfors:1954}, \cite{Hamilton:1969}, \cite{Strebel} and references therein.  A celebrated result of Teichm\"uller, which was subsequently proved using two very different methods by Ahlfors \cite{Ahlfors:1954} and by Bers \cite{Bers}, states that given any orientation preserving homeomorphism $f:S\to S'$ between two closed Riemann surfaces of genus $g>1$ there exists among all mapping homotopic to $f$, a unique extremal which minimizes the $L^{\infty}$ norm of the complex  
dilation\footnote{The dilation $K_f=\frac{|\p_z f|+|\p_{\bar z} f|} {|\p_z f|-|\p_{\bar z} f|}$.} $K(f,S)=\|K_f\|_{L^{\infty}(S)}$. 
Moreover, the extremal mapping is a Teichm\"uller map, real analytic except at isolated points and   with constant
dilation $K_f=const$. In \cite{Hamilton:1969}, Hamilton studied the extremal problem 
with a boundary data constraint, and one of his results is
a  {\em maximum principle}  of sorts stating that if $f$ is extremal,  then the maximum of its Beltrami coefficient in $S$ is the same as the maximum on $\p S$.

In higher dimensions, the problem becomes even more difficult and the references in the literature more sparse. The extremality problem  without imposing boundary conditions is studied in the landmark paper \cite{gehring-vaisala}.    Existence and uniqueness for the analogue of Gr\"otzsch problem in higher dimensions is established  in \cite{fehlmann} and a maximum principle for $C^2$ extremal  qc mappings is proved in  \cite{asadchii}. 
 More recently, in \cite{aimo},
 \cite{aim-1} and \cite{MR2314170} the  study of extremal problems for mappings of finite distortion is carried out  for   $L^p$ norms (and more general means) of the dilation functions with  $p$ finite, rather than with the $L^{\infty}$ norm.
 In the same vein, the paper \cite{Balogh-Fassler-Platis} examines extremal problems in the mean
for dilation functions based on the modulus of families of curves.

\bigskip

In the literature discussed above, the study of extremal problems for 
qc mappings in space rests on a careful analysis of compactness properties for families of qc mappings 
with a uniform bound on dilation and on techniques from geometric function theory to establish uniqueness. The finite distortion problem relies on techniques
from {\it direct methods of calculus of variations}, in which the study of the functional itself, rather than its Euler-Lagrange equations, is used. 
This approach is only natural as  the extremal problem is posed in the class of qc mappings, and so there should be no additional hypothesis 
concerning second order derivatives. With this approach, however, 
there is so little regularity that finding information about the structure of extremal mappings 
(let alone the uniqueness) has proven intractable thus far. 
In particular, there is a huge gap between the findings in the two dimensional setting vs.\ the higher dimensional theory.

\bigskip

In the present work we propose an approach to the extremal problem that is motivated by two classic papers: One by Ahlfors \cite{Ahlfors:1954}
in which an $L^p$ approximation of the $L^\infty$ distortion is used to solve the extremal  problem in the setting of Riemann surfaces. 
The other is by  Aronsson \cite{aronsson-3}, (see also \cite{aronsson-4}) where he assumes 
the extra hypothesis of $C^2$ regularity and carries out his program to determine the structure of  absolute minimizing Lipschitz extensions. 

The extremal problem for qc mappings is a $L^\infty$ variational problem that can be rephrased as follows: Given the boundary
restriction $u_0:\p \Om\to\R^n$ of a $C^1(\bar\Om,\R^n)$ qc mapping, find the qc extensions of $u_0$ to $\Om$  with minimal  trace dilation.
From this viewpoint the problem has a superficial similarity to the problem of finding  and studying {\it minimal Lipschitz extensions} $u\in Lip(\Om)$  for scalar valued functions 
$u_0\in Lip(\Gamma)$ to a neighborhood $\Gamma\subset \Om$ in such a way that $Lip(u,\Om)=Lip(u_0,\Gamma)$.\footnote{We have set  
$Lip(u,\Om)=\sup_{x,y\in \Om, x\neq y} \frac{|u(x)-u(y)|}{|x-y|}$.}  
The existence of
minimal Lipschitz extensions was settled in 1934 by McShane (see also \cite{MR2210080} for a more recent outlook  of the problem), but simple examples show that uniqueness fails. In 1967, Aronsson showed that
if the extremal condition is suitably localized to \emph{ absolute minimal Lipschitz extension} (AMLE), i.e., 
$u\in Lip(\Om)$ is AMLE with respect to $u_0\in Lip(\p\Om)$ if $Lip(u,V)=Lip(u,\p V)$ for all $V\subset \Om$,  then a $C^2$  function $u$ is AMLE if and only if it solves the 
$\infty-$Laplacian
\begin{equation}\label{inflap}
u_i u_j u_{ij}=0 \text{ in } \Om.
\end{equation}
In essence, this PDE tells us that $|\nabla u|$ is constant along the flow lines
of $\nabla u$. Aronsson also discovered several links between the geometry of
the flow lines and the regularity and rigidity properties for $\infty-$harmonic functions in planar regions. 
In the 1960's, solutions of \eqref{inflap} could only
be meaningfully defined as $C^2$ smooth. 
In the 1980's, however, a number of authors (see for instance \cite{usersguide}, \cite{jensen}) 
developed the theory of viscosity solutions, leading to Jensen's uniqueness theorem
for AMLE and for the Dirichlet problem for the $\infty-$Laplacian. 
Recent, exciting extensions of Aronsson's work to the vector-valued case provide further links with qc extremal problems 
(see Sheffield and Smart's preprint 
\cite{Sheffield-Smart}) but, as the theory of viscosity solutions has no vector valued counterpart, the standing $C^2$ hypothesis is present even in these very 
recent developments.

The similarities with the AMLE theory prompted us to study a {\it local form} of the  classical extremality condition, in which the qc mapping is required to have minimum dilation in every subset of the domain 
with respect to competitors having the same boundary values on that subset. Our goal is to find an operator that plays an
analogous role to that of  the 
$\infty$-Laplacian in the characterization of extremals and would provide a platform for the qualitative study of these mappings.
The non-linear relation between the 
dilation of a diffeomorphism and the dilation of its trace on a hypersurface introduce further complications in our work.

In order to be more specific about our  results we need to introduce some basic definitions:
If  $\phi$ is a $n\times n$ matrix of $C^1$ functions, then the {\em Ahlfors operator} $S(\phi)$ is given by
\begin{equation}
S(\phi)=\frac{\phi+\phi^T}{2}-\frac{1}{n} \tr (\phi) I  \label{eqn:Ahlfors operator}
\end{equation}
(see \cite{Ahl74}, \cite{reimann:1976}, \cite{ahlfors:1976}).
%
%
If  $u:\Omega\to \Omega'$ is a $C^1(\bar{\Omega})$ orientation-preserving diffeomorphism then 
it is quasiconformal and $\det du\ge \e>0$. For such a mapping we define the
normalized  pull back of the Euclidean metric under $u^{-1}$ as the Riemannian metric $g^{-1}$. In 
coordinates, the metric is expressed by the matrix\footnote{This metric has the following property: for all $V,W\in T_{u(x)} \R^n$ we observe that
$\langle V,W\rangle_{g^{-1}(u(x))}=
\frac{\langle du^{-1} V, du^{-1}W\rangle_{\text{Eucl}}}{(\det du^{-1})^{2/n}}.$
Hence $u: (\Om, dx^2)\to (u(\Om), g^{-1})$ is a conformal map in the sense that
$\langle du V, du W \rangle_{g^{-1}}=(\det du)^{2/n} \langle V, W\rangle_{\mathrm{Eucl}}.$
}
\begin{equation}\label{eqn: g inverse}
g^{-1}_{ij}(u(x))=\Bigg(\frac{du^{-1,T} du^{-1}}{(\det du^{-1})^{2/n}} (x)\Bigg)_{ij}=\frac{du_{ki}^{-1} du_{kj}^{-1}}{(\det du^{-1})^{2/n}} (x).
\end{equation}

The  inverse  $$g_{ij}=\frac{(du du^{T})_{ij}}{(\det du)^{2/n}} =\frac{du_{ik} du_{jk}}{(\det du)^{2/n}}.$$ In \cite{iwaniec-martin}
the metric $g$ is called {\em the distortion tensor}. 
As in the work of Ahlfors \cite{Ahlfors:1954} we consider $L^p$ approximations 

$$\inf_{v} \int_{\Om} \K_v^{np}(x) dx, $$
of the $L^\infty$ variational problem
(these approximations have been studied in depth  in \cite{aimo}).
Let $\Om\subset \R^n$ be a bounded open set.
An orientation preserving QC mapping $u:\Om\to \R^n$ is  {\em $p-$extremal}
if $\|\K_u\|_{L^p(\Om)} \le \|\K_v\|_{L^p(\Om)}$
for all orientation preserving QC mappings $v:\Om\to \R^n$ 
with $u=v$ on $\p\Om$.
It is straightforward 
to derive Euler-Lagrange equations for the $L^p$ variational problem: Every   orientation preserving 
$p-$extremal diffeomorphism $u=(u^1,...,u^n)\in C^2(\Om,\R^n)$ satisfies the fully nonlinear system of PDE
\[
(L_pu)^i=  np\p_j \Big[  \K_u^{np-2} \big( S(g)du^{-1,T}\big)_{ij} \Big] = np \p_j \Big[ \K_u^{np-2} S(g)_{\ell i} du^{j\ell}\Big]=0
\]
in $\Om$, for $ i=1,...,n$. Here $(du)^{ij}$ denotes the $ij$ entry of $du^{-1}$,  
$g^{ij}$ is defined by (\ref{eqn: g inverse}) and $S(g)$ by \eqref{eqn:Ahlfors operator}.
For $C^2$ smooth mappings with non-singular Jacobian, the operator $L_p$ can be expressed in the non-divergence form
$(L_p u)^i = A^{ik}_{j\ell } (du) u_{j\ell}^k$. The quasi-convexity of the $L^p$ variational functional \cite{iwaniec-martin} implies that
the system satisfies the Legendre-Hadamard ellipticity conditions (see 
Lemma \ref{lem:Legendre-Hadamard ellipticity}).
Motivated by  the work of Aronsson, we consider the formal limit as $p\to \infty$ of the PDE $L_p u=0$
and obtain 
\begin{equation}\label{eqn:L infty in terms of S and K}
(L_\infty u)^i = \frac{n^2  |du|^4}{\K_u^3} \big( S(g) du^{-1,T}\big)_{ij} \p_{x_j} \K_u=0,
\end{equation}
or equivalently $S(\tilde g)\nabla \K_u=0$, where $\tilde g=\frac{du^T du}{(\det du)^{2/n}}$ (see Section \ref{sec:L^infty} below). 
This PDE tells us that the trace dilation $\K_u$ is constant along the flow lines
of  the rows of the matrix $S(g) du^{-1,T}$ (and their linear combinations with $C^1$ coefficients). Since the derivation of \eqref{eqn:L infty in terms of S and K} is formal, a priori there need not be any links between solutions of this PDE and the extremal problem for qc mappings. However, such links exist and  are addressed by the main results of the present paper. 
\begin{thrm}\label{thm: max principle of K}
 Let $\Om\subset \R^n$ be an open set.
If $u\in C^2(\bar \Om,\R^n)$ is an orientation preserving diffeomorphism solution of $L_{\infty} u=0$ in $\Omega$, then for any bounded 
subdomain $\bar D\subset \Omega$,
$$\K(u,\bar D)\le \sup_{\p D} \K_u.$$
Moreover, if $n\ge 3$
and $\K_u$ has a strict maximum on $\p D$ in the sense that $\K_u(z) < \sup_{\p D} \K_u$ for $z\in D$, 
then
 \begin{equation}\label{final one}
\K(u,\bar D) = \sup_{\p D} \K_u \leq \sqrt n (n-1)^{-\frac{n-1}{2n}}  \sup_{\p D} \K_{u, \p D}^{\frac{n-1}{n}},
\end{equation}
where $\K_{u, \p D}$ denotes the dilation of the trace of $u$ on $\p D$ (see Definition \ref{tangential dilation}).
\end{thrm}
\begin{cor}\label{cor:K constant on boundary, then K constant}
Given the hypothesis of the previous theorem, 
\begin{enumerate}
\item  
if $\min_{x\in\p \Omega} \K_u(x)>\sqrt n$, then 
\[
\min_{x\in\Omega} \K_u(x) = \min_{x\in\p\Omega} \K_u(x);
\]
\item If $\K_u$ is constant with $\K_u>\sqrt{n}$ on $\p\Om$ 
then $\K_u$ is constant in $\Om$. Moreover, if $n=2$ and
$u$ is affine and  is not conformal on $\p\Om$, then $u$ is  an affine map.
\end{enumerate}
\end{cor}

\begin{thrm}\label{converse}
If $u\in C^2(\Om,\R^n)\cap C^1(\bar \Om,\R^n)$ is an orientation preserving diffeomorphism, such that  for every $\bar D\subset \Om$ and $v\in C^2(D,\R^n)\cap C^1(\bar D,\R^n)$ orientation preserving diffeomorphism with $u=v$ on $\p D$ we have
$\K(u,\bar D)\le \K(v,\bar D)$ then $L_\infty u=0$ in $\Om$. If $n\ge 3$
and  for every $D\subset \Om$,
\begin{equation}\label{final two}
\K(u,\bar D) \le n^{-\frac{1}{2n}} \sup_{\p D} \K_{u,\p D}^{\frac{n-1}{n}},
\end{equation}
then $L_\infty u=0$ in $\Om$.
\end{thrm}

\begin{cor} Let $u,v\in C^2(D,\R^n)\cap C^1(\bar D,\R^n)$ be  orientation preserving diffeomorphisms, such that  $u=v$ on $\p D$. If
 $L_\infty u=L_\infty v=0$ in $D$ then $\K(u,\bar D)=\K(v,\bar D)$.
 \end{cor}

These results echo some of the $n=2$ theory, in particular the {\it maximum principle} for the dilation in Theorem \ref{thm: max principle of K} recalls Hamilton's result \cite[Corollary 2]{Hamilton:1969}.   The  fact that the dilation is constant along flow lines of a conformally invariant set of vectors recalls the 
analogous  planar result about dilation being constant along the image
of lines under the action of the conformal mappings associated to the quadratic differentials of Teichm\"uller mappings (see  \cite[Page 175]{Strebel}  for a more detailed description).

\begin{rmrk} Theorem \ref{thm: max principle of K} and \eqref{oneway trace} tell us that if $L_\infty u=0$ in $\Om\subset \R^n$, $n\ge 3$, 
then for every $\bar D\subset \Om$ for which $\K_u$ has a strict maximum on $\p D$, $u$ is a quasi-minimizer for the extremal problem for the trace dilation in $D$.   In fact, if  $v\in C^2(D,\R^n)\cap C^1(\bar D,\R^n)$ orientation preserving diffeomorphism with $u=v$ on $\p D$,
$$\K(u,\bar D) 
\le \sqrt n (n-1)^{-\frac{n-1}{2n}}   \sup_{\p D} \K_{u, \p D}^{\frac{n-1}{n}}
= \sqrt n (n-1)^{-\frac{n-1}{2n}}   \sup_{\p D} \K_{v, \p D}^{\frac{n-1}{n}} 
\le \sqrt n (n-1)^{-\frac{n-1}{2n}}  n^{\frac{1}{2n}} \K(v,\bar D).$$
 On the other hand, Theorem \ref{converse} tells us that those diffeomorphisms that  are minimizers for the extremal problem for the trace dilation on every subset $D\subset \Om$ are also solutions of $L_\infty u=0$. This lack of symmetry in our result follows from the fact that the constants in  \eqref{final one} and \eqref{final two} are different. While the constant
in \eqref{final one} seems to be sharp, we are confident that is possible to improve 
on the constant in \eqref{final two} and conjecture: {\it If $u\in C^2(\Om,\R^n)$ then  the condition $L_\infty u=0$ in $\Om$ is equivalent to minimizing the dilation $\K(u,\bar D)\le \K(v,\bar D)$, on any subset $D\subset \subset  \Om$, among competitors $v\in C^2(D,\R^n)\cap C^1(\bar D,\R^n)$ with $v=u$ on $\p D$.} 
\end{rmrk}

A large class of solutions of $L_\infty u=0$ is provided by observing that (in any dimension) 
the set of $C^2$ solutions of $L_{\infty} u=0$ is invariant by  transformations $\tilde u=F\circ u$ and $v = u\circ F$ with $F$ conformal. In particular, all the known explicit extremal QC mappings (that we are aware of) have constant trace dilation and hence satisfy the PDE \eqref{eqn:L infty in terms of S and K}.

\begin{cor} \label{cor:L_infty examples}
(1) Any Teichm\"uller map of the form $u:=\psi\circ v \circ \phi^{-1}$
with $\psi,\phi$ conformal and $v$ affine is a solution of $L_{\infty}u=0$.  (2) the QC mappings 
$u(x)=|x|^{\al-1}x$ for $\al > 0$ solve $L_\infty u=0$ away from the origin.
(3) Let $0<\alpha < 2\pi$ and $(r,\theta,z)$ be cylindrical coordinates for $x=(x_1,...,x_n)$ where  $x_1 = r\cos\theta$, $x_2 = r\sin\theta$ 
and $x_j=z_j$, $3\leq j\leq n$.
The QC mapping 
\begin{equation}\label{eqn:wedge L_infty example}
u(r,\theta,z) = 
\begin{cases} (r,\pi\theta/\alpha,z) & 0\leq \theta\leq\alpha \\ (r,\pi+\pi\frac{\theta-\alpha}{2\pi-\alpha},z) & \alpha<\theta<2\pi \end{cases}
\end{equation}
solves $L_\infty u=0$ away from the set $r=0$.
\end{cor}

The proofs of Theorems \ref{thm: max principle of K} and \ref{converse}   rest on the analysis of the flow lines of the rows of the 
distortion tensor $S(\tilde g)$ and the geometric interpretation of $L_\infty u=0$. We show that if $u$ is not conformal on the boundary then these flow lines fill  (row by row) the open set.

The smoothness assumptions we make here are not natural for the problem, as they do not guarantee the necessary 
compactness properties that we need to prove existence of extremals. However, in the spirit of Aronsson's work on $C^2$ AMLE, it is plausible
that the study of $C^2$ mappings can yield a measure of intuition for the general setting. 

We observe that in the proof of the first part of Theorem \ref{thm: max principle of K}, the smoothness hypothesis can be decreased to $W^{2,p}$ for $p$ sufficiently high, using the work of DiPerna and Lions \cite{MR1022305}  (see also \cite{MR2096794}) on solutions of ODE with rough coefficients. In fact, we can rephrase the PDE \eqref{eqn:L infty in terms of S and K} in the following terms:  
{\it A QC mapping $u:\Om\to \Om'$ is a weak solution of $L_\infty u=0$ in $\Om$ if the trace of the corresponding distortion tensor $\tilde g$
is constant along flow lines of linear combinations of the rows of $S(\tilde g)$.}   In this formulation,  the components of $du$  need only be in a suitable Sobolev space or in BV.
At present we 
are unable 
to decrease the smoothness hypothesis to the {\it natural} category of QC mappings and still obtain the maximum principle.

Although currently we do not know how to prove existence of solutions of $L_\infty u=0$ or how to attack the extremal problems for a 
fixed homotopy class of qc mappings,  we indicate a possible strategy for a proof which involves
the construction of  a competitor for $u$ by flowing $u$ along a gradient flow for the $L^p$ norm of the dilation, then 
letting  $p\to \infty$. The initial value problem we need to control is the following:

\begin{equation}\label{intro:ivp}
\begin{cases}
\p_t u_p -L_p u_p=0 & \text{ in }Q \\
u_p=u &\text{ on }\pa Q,
\end{cases}
\end{equation}
where $Q=\Om\times (0,T)$ and $\pa \Om= \Om \times \{0\} \cup \p\Om\times (0,T)$.
We prove the  following

\begin{prop}\label{thrm:short time existence non-linear}
Let $u_0:\Omega\to \R^n$ be a $C^{2,\alpha}$ diffeomorphism, for some $0<\al<1$  with $\det du_0\ge \e>0$ in $\bar\Omega$. Assume that
 $$A_{jl}^{ik}(du_0)\p_j\p_l u_0^k=0,$$
 for all $x\in \p\Omega$ and $i=1,\dots,n$.

For every $\mu\in (0,\al)$ there exists positive constants $C>0$ depending on
 $p, n, \Omega, \e, \|u_0\|_{C^{1,\alpha}(\bar \Omega)}$, and $T>0$ depending on $p, n, \Omega, \e, \|u_0\|_{C^{2,\alpha}(\bar \Omega)}$
 and a diffeomorphism $u\in C^{2,\mu}(Q)$ solving \eqref{intro:ivp} such that
 \begin{equation}
 \|u\|_{C^{2,\mu}(Q)} + \|\p_t u\|_{C^{0,\mu}(Q)} \le C \|u_0\|_{C^{2,\alpha}(\Om)},
 \end{equation}
 \begin{equation}
 \det d u \ge \frac{\e}{2} \text{ for all }(x,t)\in Q.
 \end{equation}
\end{prop}

We remark that although flows of qc mappings have been studied and used 
several times in the literature, see for instance \cite{Ahl74}, \cite{reimann:1976}, \cite{ahlfors:1976}, \cite{MR2401620}, and \cite{MR0402041}, this is the first instance of a gradient flow used in this context. Study of this flow may also contribute to a better understanding of the well-posedness and long-time behavior of initial value problems related to gradient flows of quasi-convex (and non convex) functionals 
(see \cite{evans-savin-gangbo}).
\vspace{1pc}
\noindent{\it Acknowledgments}. It is a pleasure to thank Hans Martin Reimann and Jeffrey Rauch for their interest and encouragement for this project. 
L.\ C.\  would like to dedicate this paper  in fond  memory of Juha Heinonen, who continues to be an inspiring role model. 

%
%

\section{Preliminaries}

%
%
A map $F:\R^n\to\R^n$ is \emph{conformal} if at every point
$$dF^T dF = \lambda I_n,$$
for some scalar function $\lambda$. Liouville's theorem states that if $n>2$ then
$1-$quasiconformal mappings are conformal and that the only conformal mappings are  compositions of  rotations,  dilations, and the inversion $x\mapsto x/|x|^2$. If $n=2$, then orientation preserving conformal mappings are biholomorphisms (and vice versa).
A simple computation shows that $\lambda=|dF|^2/n$ and $\det dF=\sqrt{\lambda^n}$.
%
%
We now list some equivalent formulations of conformality.
\begin{lemma}\label{lem:conformal properties}
 Let  $F:\R^n\to\R^n$ be a diffeomorphism. The following are equivalent:
\begin{enumerate}\renewcommand{\labelenumi}{(\alph{enumi})}
\item $F$ is conformal; 
\item $\K_F=\sqrt{n}$ identically;
\item The expression $(dF)^{ji} - n \frac{ (dF)_{ij}}{|dF|^2}$ vanishes identically;
\item $S(\frac{dF dF^T}{(\det dF)^{2/n}})=0$.
\end{enumerate}
\end{lemma}
Note that  if $n=2$ and $u$ is holomorphic with $\p u/\p z \neq 0$, then
$(du)^{ji} - n \frac{ (du)_{ij}}{|du|^2}=0$ is a restatement of the Cauchy-Riemann equations. 

The action of conformal mappings on $S, \K_u$ and $g$  follows immediately from the definitions.
\begin{lemma}\label{lem: F circ u properties}
Let $u:\R^n\to\R^n$ be a diffeomorphism  and $F$ be an orientation preserving  conformal mapping. If we set $\tilde u= F\circ u$
and denote by $\tilde \K$ and $\tilde g$ the corresponding dilation and distortion tensor, then
\begin{enumerate}\renewcommand{\labelenumi}{(\alph{enumi})}
\item $\tilde \K=\K_u$;
\item $\tilde g= \lambda^{-1} dF g dF^T$;
\item $S(\tilde g)= \lambda^{-1} dF S(g) dF^T$;
\item $(d\tilde u^{-1})^T-n \frac{d\tilde u}{|d\tilde u|^2}=-n \K_u^{-2} (dF^T)^{-1}S(g) (du^{-1})^T$.
\end{enumerate}
\end{lemma}
In a similar fashion we will be interested in compositions with conformal mappings from the right, i.e., $\tilde u = u\circ F$, for which we can show:
\begin{lemma}\label{lem: u circ F properties}
Let $u:\R^n\to\R^n$ be a diffeomorphism  and $F$ be an orientation preserving  conformal mapping. If we set   $\tilde u= u\circ F$
and denote by $\tilde \K$ and $\tilde g$ the corresponding dilation and distortion tensor, then
\begin{enumerate}\renewcommand{\labelenumi}{(\alph{enumi})}
\item $\tilde \K=\K_u$;
\item $\tilde g= g$;
\item $S(\tilde g)=S(g)$;
\item $(d\tilde u^{-1})^T-n \frac{d\tilde u}{|d\tilde u|^2}=-n \K_u^{-2} S(g) (du^{-1})^T(dF^T)^{-1}.$
\end{enumerate}
\end{lemma}

%
%
\section{The Euler-Lagrange system}
Let $\Omega\subset \R^n$ be a bounded, smooth, open set and $u:\Omega\to \R^n$
a smooth, orientation preserving diffeomorphism
with $0<\ddu <\infty$.  For
 $1\le p\le \infty$,
we define, whenever the expression if finite,
$$\mathcal F_p (u,\Omega)= \frac{1}{|\Omega|} \int_{\Omega} \K_u^{np} dx .$$
For any $\psi\in C^{\infty}_0(\Omega,
R^n)$ we set 
$h(s):=\Fp(u+s\psi,\Omega)$
and compute
\begin{align}
\frac{d}{ds} h(s)\bigg|_{s=0}= & \frac{1}{|\Omega|}
\int_{\Omega} np (\ddu)^{-p} |du|^{np-2} du \cdot d\psi 
- p (\ddu)^{-p-1} \p_j \psi^i (\cofu)_{ij} |du|^{np} dx \notag \\
=&  \frac{1}{|\Omega|}p \int_{\Omega} \p_j \Bigg( \frac{ |du|^{np}}{ (\ddu)^{p+1}} (\cofu)_{ij} -
\frac{n |du|^{np-2}}{ |\ddu|^p} \p_j u^i\Bigg) \psi^i \, dx  \label{eqn:first variation}
\end{align}
where $\cofu$ denotes the cofactor matrix of $du$,  so that
$(\cofu)^T du= \ddu I$. Define the operator $L_p$ on $\R^n$-valued functions by 
\begin{align}
(L_pu)^i= & - p \p_j \bigg( \bigg[(du)^{ji} - n \frac{ (du)_{ij}}{|du|^2}\bigg] \frac{|du|^{np}}{(\ddu)^p}\bigg) \notag
=-p \p_j  \bigg(   du^{-1} \bigg[ I_n - n \frac{du du^T}{|du|^2}\bigg]  \frac{|du|^{np}}{(\ddu)^p} \bigg)_{ji} \\
&= np\p_j \Big[  \K_u^{np-2} \big(du^{-1} S(g)\big)_{ji} \Big] = np \p_j \Big[ \K_u^{np-2} S(g)_{\ell i} du^{j\ell}\Big], \label{eqn:Lp in terms of S and K}
\end{align}
where $du^{ij}$ denotes the $ij$ entry of the inverse of $du$, and $I_n$ is the $n\times n$ identity matrix, and
$\K_u$ is defined in \eqref{eqtn: analytic def},  
$g^{ij}$ by (\ref{eqn: g inverse}) and $S(g)$ by \eqref{eqn:Ahlfors operator}. Note that the equality of the first and third expressions 
in \eqref{eqn:Lp in terms of S and K} uses 
\begin{equation}\label{eqn:factoring}
(du^{-1})^T-n \frac{du}{|du|^2}= -n \K_u^{-2} S(g) (du^{-1})^T.
\end{equation}

We write $(L_p u)^i=\p_j A^i_j(du)$ where
\[
A^i_j(q)=-p\bigg[ q^{ji}-n \frac{q_{ij}}{|q|^2}\bigg] \frac{|q|^{np}}{(\det q)^p},
\] 
defined for any non-singular $n\times n$ matrix $q$.
Notice that $A_j^i(q)q_{ij}=0$. 
Set $A^{ik}_{j\ell }(q):=\frac{\p}{\p q_{k\ell }} A^i_j(q)$. Recalling that
\[
\p_{q_{k\ell }} (\cof q)_{ij}=\cof q_{k\ell } q^{ji} - \cof q_{i\ell } q^{jk} \qquad \text{ and } \qquad
\p_{q_{k\ell }} q^{ji}=- q^{\ell i} q^{jk},
\]
we compute
\begin{equation}
A^{ik}_{j\ell }(q)
=-p \frac{|q|^{np-2}}{(\det q)^p} \Bigg[ np (q_{k\ell }q^{ji}+q_{ij}q^{\ell k}) -n (np-2) \frac{q_{ij}q_{k\ell }}{|q|^2}
-|q|^2 (q^{\ell i}q^{jk} +p q^{\ell k} q^{ji}) - n \delta_{ki} \delta_{j\ell }\Bigg]. \label{eqn: A^{ik}_{jl}}
\end{equation}
For $C^2$ smooth maps with non-singular Jacobian, the operator $L_p$ can be expressed in non-divergence form:

\begin{equation}\label{non-div}(L_p u)^i = A^{ik}_{j\ell } (du) u_{j\ell}^k.\end{equation}

We remark that, in this form, the operator satisfies a  Legendre-Hadamard ellipticity condition. This result can be inferred by observing 
that the functional $\mathcal F_p (u,\Omega)$ is quasi-convex (it is actually polyconvex,  this is proved in 
\cite[Corollary 8.8.1]{iwaniec-martin}), and consequently, given sufficient smoothness, satisfies Legendre-Hadamard conditions. As we need explicit expressions for the constants involved, we provide the following estimates, 
whose elementary proof we omit.

\begin{lemma}\label{lem:Legendre-Hadamard ellipticity}
For  $n\geq 3$ and $p\geq 1$ or $n\geq 2$ and $p>1$ and for all non-singular  matrices $q$ and vectors $\xi, \eta \in \R^n$, we have
\begin{equation}\label{LHC}  C_1(n,p) p |\eta|^2 |\xi|^2 \frac{|q|^{np-2}}{(\det q)^p }\le  A^{ik}_{j\ell }(q) \eta_i \xi^j \eta_k\xi^\ell \le C_2(n) p^2 |\eta|^2 |\xi|^2 \bigg(\frac{|q|^{np-2}}{(\det q)^p}+\frac{|q|^{n(p+2)-2}}{(\det q)^{p+2}}\bigg) ,\end{equation}
where we can choose $C_1(n,p)=n$ for $n\ge 4$ and $p\ge 1$ and for $n\ge 3$ and $p>1$; $C_1(n,p)= \frac{6p-3}{p+1}$ if $n=3$ and $p\ge 1$ and
$C_1(n,p)=2\frac{p-1}{p+1}$ for $n=2$ and $p>1$.  The constant $C_2(n)$ does not depend on $p$ and can be chosen to be $C_2(n)=100n^3$.
\end{lemma}

\begin{rmrk}
The operator $L_p$ does not satisfy the stronger ellipticity condition $\Lambda |\eta|^2\ge A^{ik}_{j\ell } \eta_{ij} \eta_{k\ell }\ge \lambda |\eta|^2$. \end{rmrk}

As the dilation functional is invariant under the action of conformal mappings
(i.e., $\mathcal F_p (u,\Omega)=\mathcal F_p (F(u),\Omega)$ for all conformal mappings $F:\R^n\to\R^n$ that map $\Omega$ into itself), we 
can expect a corresponding invariance for the solutions of $L_p u=0$. 

\begin{prop}\label{Proposition conformal invariance} Let $u:\Om\to\R^n$ be a orientation preserving  diffeomorphism.
\begin{enumerate} \renewcommand{\labelenumi}{(\roman{enumi})}
\item  If $F:\R^n\to\R^n$ is
a conformal map and $\tu=F(u)$, then
$$(\tL_p \tu )^i=\bigg( [ dF^{-1}|_u]^T L_p u \bigg)^i,$$
where 
$$(\tL_p u)^i=-p \p_j  \bigg(   du^{-1} \bigg[ I_n - n \frac{du du^T}{|du|^2}\bigg]  \frac{|du|^{np}}{(\pm\ddu )^p}
\bigg)_{ji} ,$$
with the sign in the denominator being $+1$ if $F$ is orientation preserving and $-1$ otherwise. 
\item  If $F:\Om\to\Om$ a composition of  dilations, translations, and the inversion $x\mapsto x/|x|^2$, then $v = u\circ F$ satisfies
\[
(\tL_p v)^i =  (L_p u)^i \big|_F.
\]
\end{enumerate}
\end{prop}
\begin{rmrk} Case \emph{(i)} holds for all conformal mappings, including in $n=2$ all invertible holomorphic and anti-holomorphic functions. In contrast, case  \emph{(ii)} only applies to the given set of conformal transformations, as in the plane it fails to hold except for linear invertible holomorphic and anti-holomorphic functions. 
\end{rmrk}

%
%

\section{The Aronsson-Euler-Lagrange system and the operator $L_\infty$}\label{sec:L^infty}

In this section we assume that for each $p>1$ we have a solution $u_p$ of the PDE 
\begin{equation}\label{eqn:L^p=0} L_p u_p =0, \text{ in }\Om,
\end{equation}
and that $u_p\to u_{\infty}$ in $C^2$ norm on subcompacts of $\Om$. Our goal is to formally derive a system of PDE for $u_{\infty}$.

 Observe that
$\p_j |du|^{np}= np|du|^{np-2} u^k_\ell u^k_{\ell j}$ and
$\p_j |\ddu|^{-p-1} = -(p+1) |\ddu|^{-p-2} (\cof du)_{k\ell} u^k_{\ell j}$.
Using the fact that $\p_j (\cofu)_{ij}=0$ for $i=1,\dots,n$, we compute
\begin{multline}\label{eqn:L^p u expression for asymptotics}
(L_p u)^i =- p \frac{|du|^{np-4}}{|\ddu|^p} \bigg\{ 
np\frac{|du|^2}{\ddu} \Big( u^k_\ell (\cof du)_{ij} + u^i_j (\cof du)_{k\ell} \Big) \\
- (p+1) \Big(\frac{ |du|^2}{\ddu}\Big)^2 (\cof du)_{k\ell} (\cof du)_{ij}
-n(np-2)  u^k_\ell u^i_j - n|du|^2\delta_{\ell j} \delta_{ik} \bigg\} u^k_{\ell j}
\end{multline}
for all $i=1,\dots,n$.

Dividing the expression above by $p^2  \frac{|du|^{np-4}}{(\ddu)^p}$ and letting $p\to \infty$, we obtain that equation 
\eqref{eqn:L^p=0} formally converges to 
$$L_{\infty} u_{\infty}=0,$$
where 
\begin{align}
(L_{\infty} u)^i 
&= - \bigg[
 n \frac{|du|^2}{\ddu} \bigg( (\cofu)_{ij}  u^k_\ell +(\cofu)_{k\ell}  u^i_j \bigg) 
-  \bigg(\frac{|du|^2}{\ddu} \bigg)^2 (\cofu)_{ij}(\cofu)_{k\ell} 
-n^2 u^i_j u^k_\ell \bigg] u^k_{\ell j} \nonumber\\
&= (ndu_{ij}-|du|^2 du^{ji}) (n du_{k \ell}-|du|^2du^{\ell k}) \p_j du_{k\ell} .\label{eqn:L_infty factored}
\end{align}
Observe that the system does not satisfy the Legendre-Hadamard conditions. 

\begin{prop}\label{prop: example} 
\begin{enumerate}
\item Let $u(x) = |x|^{\alpha-1}x$ where $\alpha\in\R$ and $\alpha\neq 0$. Then
\[
L_\infty u(x) =0
\]
and
\[
L_p u(x) = - \bigg( \frac{n+\alpha^2-1}{\alpha^2}\bigg)^{\frac{np}2} \frac{n(\alpha^2-1)(n-1)}{(n+\alpha^2-1)\alpha} \frac{x}{|x|^{\alpha+1}},
\]
away from the origin.
\item If $u(r,\theta,z)$ is defined by (\ref{eqn:wedge L_infty example}) from Corollary \ref{cor:L_infty examples}, then
$L_\infty u=0$ in the set $r\neq 0$.
\end{enumerate}
\end{prop}
\begin{proof}
For (1), direct computation yields $\K_u(x)^2= \frac{n+\al^2-1}{\al^{2/n}}$ for all $x\neq 0$
and $$S(g)_{ij}= \frac{\al^2-1}{\al^{2/n}} \bigg(\frac{x_ix_j}{|x|^2}-\frac{\delta_{ij}}{n}\bigg).$$
The proof follows from these identities and from the definition of $L_\infty$ and $L_p$.

For (2), for the case $0\leq\theta\leq \alpha$, we have  $\det du = \pi/\alpha$ and $|du|^2 = (n-1)+\pi^2/\alpha^2$. 
The computation in the $\alpha<\theta<2\pi$ case is similar.
\end{proof}

Note that $u(x) = |x|^{\alpha-1}x$ is conformal exactly when $\alpha = \pm 1$, the only cases for which $L_p u=0$.

\section{Extremal mappings and the equation $L_{\infty}u=0$}

In this section we establish some analogues of Aronsson's results in \cite[Section 3]{aronsson-3}. 

\begin{lemma} If $u\in C^2(\Om,\R^n)$, then 
\begin{equation}\label{equation: constant along flow lines}(L_\infty u)^i = \frac{n^2  |du|^4}{\K_u^3} \big( S(g) du^{-1,T}\big)_{ij} \p_{x_j} \K_u.
\end{equation}

\end{lemma}

\begin{proof} Observe that
\begin{equation}\label{derivatives of K}
\p_{q_{ij}} \K_u = \frac{1}{n} \bigg( n \frac{du_{ij}}{|du|^2} - du^{ji}\bigg) \K_u=\K_u^{-1}\big(S(g) du^{-1,T}\big)_{ij},
\end{equation}
and $\p_{x_j} \K_u = (\p_{q_{k\ell}}\K_u) u^{k}_{j\ell}$. The result follows quickly from \eqref{eqn:factoring} and \eqref{eqn:L_infty factored}.
\end{proof}

The following proposition on conformal invariance of solutions of $L_\infty u$ follows immediately from combining previous lemma with
Lemma \ref{lem: F circ u properties} and Lemma \ref{lem: u circ F properties}.
\begin{prop}
The set of $C^2$ solutions of $L_{\infty} u=0$ is invariant by  transformations $\tilde u=F\circ u$ and $v = u\circ F$ with $F$ conformal.
\end{prop}
\begin{cor} In the plane any Teichm\"uller map of the form $u:=\psi\circ v \circ \phi^{-1}$
with $\psi,\phi$ conformal and $v$ affine is a solution of $L_{\infty}u=0$.
\end{cor}

\begin{lemma}\label{lemma: bound on S(g) in terms of K}
If $\K_u = K_0>\sqrt n$, then there exists $\e = \e(K_0)>0$ so that 
\[
\e
\leq  |S(g)|^2 \leq 
\K_u^4 \Big(1-\frac{1}{n} \Big).
\]
\end{lemma}
\begin{proof}Let $0 \leq \lambda_1 \leq \cdots \leq \lambda_n$ be the eigenvalues of $g$. 
We can write $S(g) = g - \frac{\tr g}n I$ and $\K_u^2 = \tr(g)$. Direct computation yields
\[
|S(g)|^2 = \tr \bigg( \Big[ g - \frac{\tr g}{n}I\Big]^2\bigg) = \tr(g^2) - \frac1n \tr(g)^2.
\]
Note that the upper bound for $|S(g)|^2$ is now immediate. 

We now prove the lower bound. Note that the $n$-tuple of positive numbers $(\lambda_1,\dots,\lambda_n)$ satisfy $\lambda_1 \cdots \lambda_n =1$ and
$\K_u^2 = \lambda_1 + \cdots \lambda_n \geq n$ with equality if and only if $\lambda_1 = \cdots =\lambda_n =1$. 
Set $\bar\lambda = \K_u^2/n = \frac 1n(\lambda_1 + \cdots \lambda_n)$.
Since $\tr(g^2) = \sum_{i=1}^n \lambda_i^2$, it follows that
\[
\sum_{i=1}^n \big(\lambda_i - \bar\lambda\big)^2 = \sum_{i=1}^n\lambda_i^2 - 2\bar\lambda\sum_{i=1}^n\lambda_i + n\bar\lambda^2
= \tr(g^2) - \frac1n \tr(g)^2 = |S(g)|^2.
\]
We now claim that for $\delta>0$, there exists $\e = \e(\delta)$ so that whenever  
$\K_u^2/n = \bar\lambda \ge 1+\delta$ then  $\sum_i (\lambda_i - \bar \lambda)^2 \ge \e$.

To prove the claim, we argue by contradiction. Assume that there exists $\delta_0>0$
such that for each $k\in\N$ we can find positive $\lambda^k_i$ as in the hypothesis
with $\bar\lambda^k-1\ge \delta_0 >0$ and
\begin{equation}\label{equation contradiction}
\sum_i (\lambda_i^k - \bar \lambda^k)^2 \le \frac{1}{k}.\end{equation}
 
 If $\bar \lambda^k$ is a bounded sequence then
so are $\lambda^k_i$ (as $\lambda^k_i\geq 0$), 
hence for an appropriate subsequence we may assume that  $\bar \lambda^k\to\bar \lambda\ge 1+ \delta_0 $ and 
$\lambda_i^k\to \lambda_i$ as $k\to \infty$. As $\lambda^k_1\cdots\lambda^k_n=1$ for all $k$, it follows that $\lambda_1\cdots\lambda_n=1$ and 
$\lambda_i>0$. From \eqref{equation contradiction} we conclude that 
$\lambda_i=\bar \lambda$ and $1=\lambda_1\cdots\lambda_n=\bar \lambda^n\ge (1+ \delta_0)^n$, a contradiction.

 If $\bar \lambda^k$ is an unbounded sequence then
 for each $M>0$ there exists $\ell=\ell_M>0$ such that
 $\bar \lambda^\ell\ge M$. On the other hand, in view of \eqref{equation contradiction} we have $\lambda_i^\ell \ge M/2$ and consequently
 $1=\lambda_1^\ell \cdots \lambda_n^\ell \ge (M/2)^n$, a contradiction.
\end{proof}
\begin{rmrk}
When $n=2$, we can find an explicit lower bound. In this case, $\lambda_1\lambda_2=1$ and 
\[
|S(g)|^2 = \lambda_1^2 + \lambda_2^2 - \frac12 (\lambda_1+\lambda_2)^2 
= \frac12 (\lambda_1 + \lambda_2)^2 - 2\lambda_1\lambda_2 = \frac12(\K_u^4 - 4).
\]
\end{rmrk}

We are now ready to study the relation between $C^2$ extremal quasiconformal mappings and the operator $L_\infty$.

%


\begin{prop}\label{prop: max principle of K}
 Let $\Om\subset \R^n$ be an open set.
If $u\in C^2( \Om,\R^n)$ is an orientation preserving diffeomorphism solution of $L_{\infty} u=0$ in $\Omega$ then for any bounded sub-domain $\bar D\subset \Omega$, 
$$\sup_{D} \K_u \le \sup_{\p D} \K_u.$$
\end{prop}

\begin{proof}
Let $\mu=\sup_{\p D} \K_u$ and assume that there exists $p_0\in D$ such that
$\K_u(p_0)=k_0>\mu\ge\sqrt{n}$.
Since $u\in C^2$  and $\det du\ge \e>0$, $S(g)(du^{-1})^T$ is  Lipschitz in $\bar D$. Consequently,
for each $p_0\in D$ and  $i=1,..,n$ there exists a unique trajectory $\gamma_i(s)$ defined for $s\in I\subset \R$ through $p_0$ satisfying
$\frac{d}{ds} \gamma_i^j(s)=[S(g)(du^{-1})^T]_{ij}(\gamma_i(s))$ for $j=1,...,n$. Using  (\ref{equation: constant along flow lines})
and the fact that $L_\infty u=0$, we have
\[
\frac{d}{ds} \K_u(\gamma_i(s)) = \frac{d}{ds} \gamma^j_i(s) \p_{x_j}\K_u(\gamma(s)) = S(g)(du^{-1})^T_{ij} \p_{x_j}\K_u(\gamma(s))=0,
\]
so
\[
\K_u(\gamma_i(s))=\K_u(p_0)
\]
for all $s\in I$ and all $i=1,\dots,n$. 
If a curve $\gamma_i$ terminates at a point $p$ inside $D$, than at $p$ there must exist another flow curve   $\gamma_l$ that flows out of it. 
In fact, not all $\gamma_i$ can have vanishing speed simultaneously at a point inside $D$. 
Arguing by contradiction, if this were to happen then  we would have
$S(g)=0$ at the end point. This would would yield $\K_u=\sqrt{n}$ at the end point, while
$\K_u(\gamma_i(s))=k_0>\sqrt{n}$, a contradiction. We choose $i$ so that 
$$\sup_j\Big | S(g)_{ij}(\gamma_i(s))\Big| \ge  C_n | S(g) |>0,$$
for $C_n \ge \frac{1}{n^2}$. 

The argument yields a piecewise $C^1$ curve $\gamma$ inside $D$, passing through $p_0$   with $\K_u(\gamma(s))=k_0$ with
$$\frac{d}{ds} \gamma^j(s)=[S(g)(du^{-1})^T]_{ij}(\gamma(s))$$ for some index $i=1,...,n$ and 
\begin{equation}\label{eqn lower bound}
\sup_j\Big | S(g)_{ij}(\gamma(s)) \Big | \ge C_n \big| S(g) \big|
\end{equation}
for all $s\in I$. There are two alternatives: (i) the curve $\gamma$ has finite length and so  touches the boundary $\p D$ in two points $P,Q\in D$; (ii) The curve $\gamma$ does not touch $\p D$ and so has infinite length.

In (i), it follows that  $\K_u(P)=k_0> \sup_{\p D} \K_u\ge \K_u(P)$, a contradiction that $k_0>\mu$.

We need to exclude the second alternative. For simplicity we assume that the composition of flow lines is actually one single flow line, 
the general case is proved in the same way. For each $i=1,...,n$, we have 
\begin{align}
u^i(\gamma(t))- u^i(p_0) &= \int_0^t \frac{d}{ds} u^i (\gamma(s)) ds \notag \\
\text{(for some }l=1...,n) \ \ &= \int_0^t \bigg[S(g) (du^{-1})^T\bigg]_{lj}(\gamma(s)) du_{ij}(\gamma(s)) ds 
\notag \\
&= \int_0^t S(g)_{li}(\gamma(s)) ds 
\end{align}
Consequently, for some $0\ge t_l \ge t$,
$$\sup_{i=1,...,n} |u^i(\gamma(t)) - u^i(p_0)| \ge t \sup_{i=1,...,n} |S(g)_{li}|(\gamma(t_l)),$$ and by  \eqref{eqn lower bound}, we conclude
$$\sup_{i=1,...,n} |u^i(\gamma(t)) - u^i(p_0)| \ge C_n |S(g)|(\gamma(t_l))  t.$$
Since $|S(g)|$ is bounded from below by Lemma \ref {lemma: bound on S(g) in terms of K},
$|u(\gamma(t))-u(p_0)|$ has at least linear growth. Consequently,  if
 $\gamma$ has infinite length, then $u(D)$ would have to be unbounded, whereas
 since $D$ is bounded so is $u(D)$.
\end{proof}

We can now prove Corollary \ref{cor:K constant on boundary, then K constant}.
\begin{proof}[Proof of Corollary \ref{cor:K constant on boundary, then K constant}.] 
Using the argument from the previous proof, we have that the set of points 
$x\in \Om$ with $\K_u(x)>\sqrt{n}$ can be covered by compositions of  flow lines of the rows of $S(g)(du^{-1})^T$ with 
 $\K_u$  constant along these curves. We have shown that if $\K_u>\sqrt{n}$ on such a curve then  it must reach the boundary $\p\Om$. To prove (1) we observe that for any $\e>0$ such that $\K_u>\sqrt{n}+\e$ on $\p\Om$, if  $x_0\in\{ x\in \Om | \K_u(x)\in (\sqrt{n},\sqrt{n}+\e)\}$
 then there exists a composition of flow lines passing through $x_0$
 which must reach the boundary and hence contradict the hypothesis $\K_u>\sqrt{n}+\e$ on $\p\Om$. 
 As for  (2), we observe that by virtue of (1) every point in $\Om$ can be connected to the boundary with a composition of flow lines along which $\K_u$ is constant, thus concluding the proof. \end{proof}
 \begin{rmrk}
 Arguing as in the proof of \eqref{equation: constant along flow lines}, we can show that for each $i=1,...,n$, if we let $\gamma:[0,\e)\to \Om$ be a flow line
 of the $i$-th row of $S(g)du^{-1,T}$, then for any $j=1,...,n$ and $0<t<\e$ we have
 $$du_{ij}(\gamma(t))- du_{ij}(\gamma(0))= \int_0^t\frac{d}{ds} u^i_j(\gamma(s)) ds= \int_0^t \dot{\gamma}^k u^i_{jk}(\gamma(s)) ds$$
 $$=  \int_0^t \big(S(g)du^{-1,T}\big)_{ik} u^i_{jk}(\gamma(s)) ds=\int_0^t \K_u \p_{x_j} \K_u (\gamma(s)) ds .$$
 This formula allows us to recover the differential of $u$ from the dilation and the flow lines of the distortion tensor. In particular,
 if $\K_u$ is constant in $\Om$ then the rows of $du$  are constant along the flow lines of  the corresponding  rows of $S(g)du^{-1,T}$.
 \end{rmrk}
 
The previous remark yields: 
\begin{prop}\label{planar case} In the hypothesis of the previous theorem, if $\Om\subset \R^2$ and $du$ (and hence $\K_u$) is constant in $\p\Om$, with $\K_u>\sqrt{n}$ on $\p \Om$, then $du$ is constant in $\Om$ and hence $u$ is affine.
\end{prop}
\begin{proof}
The remark above implies that if $\K_u$ is constant then the rows of $du$ are constant along the flow lines of the corresponding  rows of $S(g)du^{-1,T}$. It suffices then to show that for every point $p_0\in \Om$ we can find  flow lines
of both rows of $S(g)du^{-1,T}$ passing through that point and touching the boundary $\p\Om$. To establish this fact we recall that $|S(g)|>0$ in $\Om$
and that, since we are in the planar case, both rows of $S(g)$ cannot vanish unless they vanish simultaneously, which is impossible. Since $du$ is invertible
the rows of   $S(g)du^{-1,T}$ cannot vanish at any point in  $\Om$. Repeating the argument in the proof of  Theorem \ref{thm: max principle of K} we see that the flow lines of the two rows of $S(g)du^{-1,T}$ through $p_0$ cannot end in $\Om$, nor can they continue for an infinite time, hence they must reach the boundary in a finite time.
\end{proof}
\begin{rmrk} If $n=2$ and $\K_u>\sqrt{n}$ on $\p\Om$ then $L_\infty u=0$ actually implies that $\K_u$ is constant along any path in $\Om$. Hence,
in the plane there will be no $C^2(\Om,\R^2)\cap C^1(\bar\Om,\R^2)$ solutions of $L_\infty u=0$ in $\Om$ unless $\K_u|_{\p\Om}=const$.
\end{rmrk}

We conclude this section with the proof of the necessity of the condition
$L_\infty u=0$ for a $C^2$ qc mapping to locally minimize dilation in subsets $D\subset \Om$, among competitors with the same dilation on $\p D$.

\begin{prop}\label{prop:converse}
Let $u\in C^2(\Om,\R^n)$ be an orientation preserving diffeomorphism which does not solve $L_\infty u=0$ in a closed ball $\bar D\subset \Om$.  There exists 
$v\in C^2(\bar D,\R^n)$ orientation preserving diffeomorphism with $u=v$ on $\p D$ such that $\K(v,\bar D)<\K(u,\bar D)$.
\end{prop}
\begin{proof} Let $u\in C^2(\Om,\R^n)$ be an orientation preserving diffeomorphism which does not solve $L_\infty u=0$ in a closed ball $\bar D\subset \Om$. 
In view of the conformal invariance of the PDE we can assume without loss of generality 
that $D=B(0,1)$. Let $E=\{ x\in \bar D| \K(u,\bar D)=\K_u(x)\}$.  Since  $\nabla_x \K_u=0$
at any interior point $x\in E$, we must have $E\subset \p D$ 
and consequently $\K(u,\bar D)=\sup\{\K_u(x) |x\in \p\Om\}$. If $\vn$ denotes the outer unit normal at $x$ to $\p D$, then
the latter yields that $\nabla_x \K_u (x)= \al \vn$ for some $\alpha>0$  at each $x\in E$. The identity 
(\ref{equation: constant along flow lines}) then implies
\begin{equation}\label{absurd}
 S(g) du^{-1,T} \vec n \neq \vec{0}.
\end{equation}For $\lambda\in \R$ and $\chi\in C^2(\bar D,\R^n)$,  vanishing on $\p D$,  we define $u_\lambda(x)=u(x)+\lambda \chi(x)$. 
Using \eqref{derivatives of K} and a  Taylor expansion of $\K_{u_\lambda}$ in $\lambda$,
we have that 
\begin{equation}\label{competitor}
\K_{u_\lambda}= \K_u+ \lambda \p_{q_{ij}} \K_{u} d\chi_{ij} + O(\lambda^2)=
\K_u + \lambda \K_u^{-1}\big(S(g) du^{-1,T}\big)_{ij} d\chi_{ij} +O(\lambda^2).
\end{equation}
We claim that given $u$ satisfying \eqref{absurd}, we can find a mapping $\chi\in C^2(\bar D,\R^n)$, vanishing on $\p D$,
such that the coefficient of $\lambda$ in \eqref{competitor} 
is strictly negative in a neighborhood $U$ of $E$, for small values of $\lambda$.
This fact would allow us to the conclude the proof of the proposition. Indeed, for $x\in U\cap D$ and small values of $\lambda$,  we would have 
$\K_{u_\lambda}<\K_u\le \K(u,\bar D)$. On the other hand, for $x\in D\setminus U$, there 
would exist $\e>0$ such that   $\K_u<  \K(u,\bar D)-\e$, thus yielding that 
$\K_{u_\lambda} < \K(u,\bar D)-\e + C\lambda\le \K(u,\bar D)$ for small values of $\lambda$ and   $C=C(\|u\|_{C^1}, \|\chi\|_{C^2}, D)$.  
Given such inequalities we would then  conclude that $v= u_\lambda$ 
is a  qc diffeomorphism with the same boundary data as $u$  and strictly smaller dilation $\K(u_\lambda, \bar D)< \K(u,\bar D)$.

To find $\chi$, observe that if  $p\in E$ then as a consequence of \eqref{absurd} there exists $\vec v \in \R^n$ such that 
\begin{equation}\label{competitor5}
\langle S(g) du^{-1,T}\vec n, \vec v\rangle >0
\end{equation}
in a neighborhood $B(p,r)$. Since we can cover $E$ with a finite set of such neighborhoods, we obtain vectors $\vec v_1,...,\vec v_k \in \R^n$ for which
\eqref{competitor5} holds in $B(p_k,r)$ and such that $E\subset \bigcup_{l=1}^k B(p_l,r)$. For each $l=1,...,k$, let $\phi_l:S^{n-1}\to \R$ 
be a positive smooth function such that  $\phi=0$ outside $B(p_l,r)\cap S^{n-1}$. We set 
\begin{equation}\label{competitor3}
\chi(x)= (1-|x|^2) \bigg[ \sum_{l=1}^k \phi_l\Big(\frac{x}{|x|}\Big) \vec v_l\bigg].
\end{equation}
Clearly this mapping vanishes on $\p D$ and it can be easily modified near the origin to yield a smooth mapping in $\bar D$. Observe that at every point in $S^{n-1}$,
$$d\chi=- 2 \bigg(\sum_{l=1}^k \vec v_l\phi_l \bigg)  \otimes \vec n .$$
Substituting the latter in \eqref{competitor} we obtain that for every point in $\p D$, 
\begin{equation}\label{competitor4}
 \K_{u_\lambda} = \K_u -2\lambda \K_u^{-1}\big(S(g) du^{-1,T}\big)_{ij} \vec n_j \vec v_{l,i} \phi_l +O(\lambda^2)=\K_u-2\lambda\K_u^{-1}
 \sum_{l=1}^k \langle S(g) du^{-1,T}\vec n, \vec v_l\rangle \phi_l+O(\lambda^2).
 \end{equation}
 In view of \eqref{competitor5} and the choices of $\phi_l$ and $\vec v_l$, it follows
 that for all $x$ in a neighborhood   $  E\cap B(p_l,r)\subset B(p_l,r)\cap \p D$ and $\lambda$ sufficiently small, the coefficient of $\lambda$ is strictly negative
 as  $$-2\K_u^{-1} \sum_{l=1}^k \langle S(g) du^{-1,T}\vec n, \vec v_l\rangle \phi_l <0.$$
 Thus, the  strict inequality $\K_{u_\lambda}< \K_u$ holds, whereas elsewhere in $\p D\setminus \cup_{l=1}^k B(p_l,r)$ we have equality. 
 \end{proof}
 
 %
 %
 \section{Dilation of traces of diffeomorphisms}
 
 Throughout this section $\Om\subset \R^n$ is an open set, $n\ge 3$, 
 $u\in C^2(\Om,\R^n)$ is an orientation preserving diffeomorphism,  
 $M\subset \Om$ and $M'=u(M)$ are  closed, $C^1$ hypersurfaces endowed with metrics induced by the Euclidean metric. For $x\in M$,
we denote by  $e_1,...,e_{n-1}$  an orthonormal basis of $T_xM$ and by $w_1,...,w_{n-1}$ an orthonormal basis of $T_{u(x)}M'.$ 
 We let $w_0$ be the  unit normal field to $M'$ such that $\langle du\vec n, w_0\rangle>0$. We denote by $$U=u|_M$$ the trace of $u$ on $M$.
 For each $x\in M$ consider the $(n-1)\times (n-1)$ matrix $d^M U(x)=(d_{ij})$ with
 $d_{ij}= \langle du e_i, w_j\rangle^2$.
 
 \begin{dfn}\label{tangential dilation} The {\it tangential dilation} of $U=u|_M$ at a point $x\in M$ is given by
 \[
 \K_{u,M}(x)= \frac{|d^MU|}{[\det d^Mu  ]^{\frac{1}{n-1}} }.
 \]
 \end{dfn}
 
 If $v\in C^1(\Om,\R^n)$ is an orientation preserving diffeomorphism with $u=v$ on $M$ then $\K_{u,M}=\K_{v,M}$ on $M$. 
 The following lemma is  probably well known but we give a short proof as we did not find it in the literature.
 \begin{lemma}\label{lem:trace vs dilation}
 For every $x\in M$, the dilation
 \begin{equation}\label{oneway trace}
 \K_{u,M}^2\le n^{\frac{1}{n-1}} \K_u^{\frac{2n}{n-1}} - \frac{|du \vec n|^2 \langle du\vec n , w_0\rangle^{\frac{2}{n-1}}}{[\det du]^{\frac{2}{n-1}} }.
 \end{equation}
 \end{lemma}
 \begin{proof}
 We consider the two orthonormal frames of $\R^n$ given by 
 $$\{ \vec n , e_1,...,e_{n-1}\} \text{ and } \{w_0, w_1,...,w_{n-1}\}$$
 and observe that in these frames $du(x)$, $x\in M$,  can be represented as a block matrix
 $$du=\bigg(\begin{array}{ll} \langle du\vec n, w_0\rangle & 0 \\ \langle du \vec n , w_i\rangle & d^MU \end{array} \bigg).$$
 Consequently, 
 $$|du|^2=|d^MU|^2+|du \vec n|^2 \text{ and }\det du=\langle du\vec n, w_0\rangle \det d^MU.$$ The estimate \eqref{oneway trace} then follows from 
 the latter and from recalling $ \langle du\vec n, w_0\rangle \le \sqrt{n} |du|$.
 \end{proof}
 
Lemma \ref{lem:trace vs dilation} and Proposition \ref{prop:converse} immediately yield
 \begin{prop}\label{prop:converse1}
If  $u\in C^2(\Om,\R^n)$ is an orientation preserving diffeomorphism that does not solve $L_\infty u=0$ in a ball $D\subset \Om$ then $$ n^{-\frac{1}{2n}} \sup_{\p D} \K_{u,\p D}^{\frac{n-1}{n}}<\K(u,\bar D).$$
\end{prop}

Theorem \ref{converse} now follows from Propositions \ref{prop:converse} and \ref{prop:converse1}.
 
We now turn to the final step in the proof of  Theorem \ref{thm: max principle of K}.  In order  to estimate the dilation of the extension of $u|_M$ in terms of the tangential dilation we need more information about the extension.
 
 \begin{lemma}
 Let $u$ be a solution of $L_\infty u=0$ in a neighborhood of $M$. If 
 $x\in M$ satisfies $\nabla \K_u(x)\neq 0$ and  $\vec n \parallel \nabla \K_u(x)$, then 
 \begin{equation}\label{the other way trace}
\frac{n-1}{n^{\frac{n}{(n-1)}}}  \K_u^{\frac{2n}{n-1}} (x)=\K_{u,M}^2(x).
 \end{equation}
 \end{lemma}
\begin{proof}
We observe that $L_\infty u=0$ at $x$ is equivalent to
\begin{equation}\label{one trace}
du^T du \, \vec n - \frac{1}{n} |du|^2 \vec n=0,
\end{equation}
at $x$. In particular, $\vec n$ is an eigenvector of $du^T du(x)$ with eigenvalue $|du|^2/n$. Representing $du^T du$ in the orthonormal frame $\vec n, 
e_1,...,e_{n-1}$, with $e_i$  eigenvectors of $du^T du(x)$, tangent to $M$ corresponding to eigenvalues $\lambda_i^2$, $i=1,...,n-1$, 
we have the diagonal matrix 
\begin{equation}\label{two trace}
du^T du(x)= \Bigg( \begin{array}{llll} \frac{|du|^2}{n} &0 &... &0\\ 0 & \lambda_1^2 &... &0\\0 &0 &... &0 \\ 0& 0 &... &\lambda_{n-1}^2\end{array}\Bigg).\end{equation}
We remark that $du^Tdu|_{T_xM}=d^MU^T d^M U$, so that $|d^MU|^2=\sum_{i=1}^{n-1} \lambda_i^2$.
From \eqref{two trace}, we immediately obtain
$$|du|^2=\tr{du^T du}=\frac{|du|^2}{n} + \sum_{i=1}^{n-1} \lambda_i^2=\frac{|du|^2}{n} +|d^M U |^2$$
and $$\det du^2=\det(du^T du)= \frac{|du|^2}{n} (\det d^M U)^2.$$
To conclude, we have
$$\K_{u,M}^2=\frac{|du|^2-\frac{|du|^2}{n}}{(\det du)^{\frac{2}{n-1}}}|du|^{\frac{2}{n-1}} n^{-\frac{1}{n-1}}=
\bigg(1-\frac{1}{n}\bigg)n^{-\frac{1}{n-1}} \frac{|du|^{\frac{2n}{n-1}}}{(\det du)^{\frac{2}{n-1}}}.$$
\end{proof}

\begin{rmrk} It is interesting to compare  these conclusions with  the example $u(x)=|x|^{\al-1}x$ on $\p B(0,1)$. 
In this case, $\K_u^2=\frac{n-1+\al^2}{\al^{2/n}}$ and $\K_{u,\p B(0,1)}^2= n-1$. Note that the proof above does not apply as $\nabla \K_u=0$.  
\end{rmrk} 
\begin{proof}[Proof of  Theorem \ref{thm: max principle of K}] The first part of Theorem \ref{thm: max principle of K} is proven in Proposition
\ref{prop: max principle of K}. For the second statement, 
observe that if $x\in \p D$ satisfies $\K(u,\bar D)=\sup_{\p D} \K_u= \K_u(x)$ then either 
$\nabla \K_u(x)=0$ or it must be normal to $\p D$. If $\nabla_x\K_u(x)=0$, then the point $x$ 
must be a local maximum of $\K_u$ in $\Om$. Consequently, there must exist a continuum $F$ through $x$ on which $\K_u$ is constant 
and with $F\cap D\neq 0$, otherwise $x$ would be an isolated strict maximum point, 
an impossibility by the first part of Theorem \ref{thm: max principle of K}.  
However the existence of points in $D$ for which $\K_u=\sup_{\p D} \K_u$ contradicts the hypothesis 
$\K_u(z)<\sup_{\p D} \K_u$ for $z\in D$ and hence  $\nabla \K_u(x)\neq 0$. The proof now  follows immediately from Proposition \ref{prop: max principle of K} and from \eqref{the other way trace}. \end{proof}
%

%
%
\section{Quasiconformal Gradient Flows}
For a fixed  diffeomorphism $u_0:\Omega\to \R^n$,  we want to study diffeomorphism solutions
$u(x,t)$ of the initial value problem \eqref{intro:ivp}.
If there is a $T>0$ such that a solution $u\in C^{2}(\Om\times(0,T))$
exists with $\det du>0$ in $\Om\times (0,T)$, then  by the same computations as in \eqref{eqn:first variation},
\[
\frac{d}{dt} \Fpo = - \bigg( \frac{1}{|\Omega|} \int_{\Omega} | L_pu|^2 dx \bigg) \le 0,
\]
meaning that the $p$-distortion is nonincreasing along the flow.
Hence we obtain
\begin{prop}\label{proposition: non increasing energy}
If $u\in C^2(\Omega\times [0,T), \R^n)\cap C^1(\bar\Omega\times [0,T), \R^n)$ is a solution of \eqref{intro:ivp}
with $\det du >0$ in $\bar\Om\times[0,T)$, then for all $0\le t<T$, 
$\|\K_{u_p}\|_{L^p(\Om)}^p = \|\K_u\|_{L^p(\Om)}^p- \int_0^T \|L_p u(\cdot, t)\|_{L^2(\Om)} dt $ and consequently
\begin{equation}\label{eqtn: monotonicity of the energy}
\| \K_u\|_{L^p(\Om\times\{t\})} \le
\| \K_{u_0}\|_{L^p(\Om)} .
\end{equation}
\end{prop}
\bigskip
 By Lemma \ref{lem: F circ u properties} and Lemma \ref{lem: u circ F properties}, 
the functional $\Fpo$ is invariant by conformal deformation. Therefore, if we let 
$s\mapsto F_s:\R^n\to \R^n$ be a one-parameter semi-group of conformal transformations, 
then solutions to the PDE system
$$\p_t u = L_p u + \frac{d}{ds} F_s(u)\bigg|_{s=0}$$
would also satisfy  \eqref{eqtn: monotonicity of the energy}.
Recall that the flow $F_s$ is conformal if 
$$S(d\mathcal D)=\frac{ d\mathcal D + d\mathcal D^T}{2} - \frac{1}{n} \text{ trace }(d\mathcal D) I_n =0$$
where $\mathcal D=(\frac{d}{ds} F_s)\bigg|_{s=0}\circ F_0^{-1}=(\frac{d}{ds} F_s)\bigg|_{s=0}$ and  $S$ denotes the Ahlfors operator.
If $n=2$ then this amounts to $\p_{\bar z} \mathcal D=0.$ If $n\ge 3$ there is more rigidity and conformality requires that
$$\mathcal D(x)=a + Bx +2 (c \cdot x) x- |x|^2 c$$ for any vectors $a,c$ and matrix $B$ with $S(B)=0$ (see \cite{sarvas}).

We observe that in light of Proposition \ref{Proposition conformal invariance},
if $u(x,t)$ is a solution of \eqref{intro:ivp} in $\Omega\times (0,T)$  and 
$v(x,t)= \delta u(\lambda x, \delta^{-2}t)$ for some $\lambda, \delta >0$, then $v(x,t)$
is still a solution with initial data
$v_0(x)= \delta u_0(\lambda x)$ in an appropriately scaled domain. Applying inversions in a suitable way will also yield new solutions from $u(x,t)$.

\subsection{Short-time existence from smooth initial diffeomorphisms}
Throughout this section 
 $\Om\subset \R^n$ is  a bounded, $C^{2,\al}$ smooth, open set. 
 
 \begin{dfn}\label{dfn:parabolic Holder classes}
Let $\Omega\subset \R^n$ be a smooth bounded domain and for $T>0$
let $Q= \Om\times(0,T)$. The \emph{parabolic boundary} is defined by $\pa Q=
(\Om\times\{t=0\})\cup ( \p\Omega \times (0,T))$. The \emph{parabolic distance}
is $d((x,t), (y,s)):= \max ( |x-y|, \sqrt{|t-s|})$.
For 
$\alpha\in (0,1)$ we define the \emph{parabolic H\"older class}
$C^{0,\al}(Q):=\{ u\in C(Q,\R) | \  \|u\|_{C^\al(Q)}:=[u]_\al+\|u\|_0<\infty\}$, where
$$[u]_{\al}:= \sup_{(x,t),(y,s)\in Q \text{ and } (x,t)\neq (y,s)} \frac{|u(x,t)-u(y,s)|}{d((x,t),(y,s))^\al}$$
and $|u|_0=\sup_Q |u|.$
For $K\in \N$ we let $C^{K,\al}(Q)=\{u:Q\to \R| \ \p_{x_{i_1}}\cdots\p_{x_{i_K}} u\in C^{0,\al}(Q)\}$.
\end{dfn}

 \begin{prop}\label{prop:iteration}
 Let $u_0:\Omega\to \R^n$ be a $C^{2,\alpha}$ diffeomorphism for some $0<\al<1$  with $\det du_0\ge \e>0$ in $\bar\Omega$. Assume that
 $L_p u_0=0$
 for all $x\in \p\Omega$.
 There exist constants $C,T>0$ depending on
 $p, n, \Omega, \e, \|u_0\|_{C^{1,\alpha}(\bar \Omega)}$,
 and a sequence of diffeomorphisms $\{u^h\}$ in $C^{2,\al}(Q)$ with $Q=\Om\times(0,T)$ so that 
\begin{enumerate}\renewcommand{\labelenumi}{(\alph{enumi})}
\item $\det u^h \ge \frac{\e}{2}$  for all $(x,t)\in Q$,
\item $ \|u^h\|_{C^{2,\al}(Q)} + \|\p_t u^h\|_{C^{0,\al}(Q)} \le C\|u_0\|_{C^{2,\al}(\Om)}$,
\item $\begin{cases}
\p_t u^{h,i} -A_{jl}^{ik}(du^{h-1}) \p_j \p_l u^{h,k}=0 & \text{ in }Q \\
u^h=u_0 &\text{ on }\pa Q.
\end{cases}$
\end{enumerate}
 \end{prop}
%
 \begin{proof} 
 We prove the result by induction. We start with  the base case $h=0$.
 Since
$u_0\in C^{2,\al}(\Omega)$, if we set $a_{jl}^{0,ik}(x):=A_{jl}^{ik}(du_0(x))$
then $a_{jl}^{0,ik}\in C^{1,\al}(\Omega)$ and in view of Lemma \ref{lem:Legendre-Hadamard ellipticity}, $a_{jl}^{0,ik}$  satisfies 
for all $\xi,\eta\in \R^n$ and $x\in \Omega$
 \begin{equation}\label{eqn:LH ellipticity for a_jl^ik h}
 \lambda^h |\xi|^2 |\eta|^2 \le a_{jl}^{h,ik}\xi^i\xi^k\eta_j\eta_l  \le \Lambda^h |\xi|^2 |\eta|^2,
 \end{equation}
with $h=0$ and 
\begin{equation}\label{eqn:Big Lambda bounds h}
\Lambda^h= C_2(n) p^2  \bigg(\frac{|du_h|^{np-2}}{(\det du_h)^p}+\frac{|du_h|^{n(p+2)-2}}{(\det du_h)^{p+2}}\bigg) 
\le C_2(n)p^2 \frac{(\|du_0\|_{L^{\infty}(\Omega)})^{n(p+2)-2}}{\e^{p+2}}
\end{equation}
and
\begin{equation}\label{eqn:small lambda bounds h}
\lambda^h= C_1(n,p) p \frac{|du_h|^{np-2}}{(\det du_h)^p} \ge\frac{ C(n) }{C_h\|du_0\|_{L^{\infty}(\Omega)}}
\end{equation}
with $h=0$ and  $C_0=1$. We have also used the bound $\det q \le n|q|^n$.
 
Applying Lemma \ref{lemma: existence and regularity for linear systems} with $T_0=1$
we obtain a constant $C_0=C_0(n,p,\e,\|du_0\|_{C^{\al}(\Om)})>0$ and a map $u^1\in C^{2,\al}(Q)$ that solves (c)  and satisfies 
\begin{equation}\label{alas}
T^{\al/2} \Big(T^{-1/2}\|du^h\|_{C^\al(Q)}+ \|u^h\|_{C^{2,\al}(Q)} + \|\p_t u^h\|_{C^{0,\al}(Q)}\Big) 
\le C_{h-1} \|u_0\|_{C^{2,\alpha}(\Om)},
\end{equation}
for $h=1$.

Next, the bound on $T$ will be imposed to keep the determinant from vanishing.
Set $w=u^1-u_0$ and observe that this map solves the equation
 \begin{equation}\label{pde-h=1 intermediate}
\Bigg\{
\begin{array}{ll}
\p_t w^i -A_{jl}^{ik}(du^{0}) \p_j \p_l w^i=A_{jl}^{ik}(du^{0}) \p_j \p_lu_0^i & \text{ in }Q \\
w=0 &\text{ on }\pa Q
\end{array}
\end{equation}
for $i=1,\dots, n$.
An application of the Schauder estimates (\ref{eqn:scaling}) yields
\begin{equation} \label{eqn:scaled bounds for dw}
T^{-\frac{1-\al}2} \|dw\|_{C^\al(Q)} \le C(n,p,\e,\|du_0\|_{C^{\al}(\Om)}) \|u_0\|_{C^{2,\alpha}(\Om)}.
 \end{equation}
Choose $T_1<1$ sufficiently small depending only on $n$, $p$, $\alpha$, $\|du_0\|_{C^{\al}(\Om)}$ and $\e=\min_{\Om} \det du_0$ so that
$\|dw\|_{C^{\al}(Q)} \ll  \frac{\e}{2}$. Since the determinant has polynomial dependence on the coefficients,
we have (a) for $h=1$ in
$Q_1 = \Omega \times (0,T_1)$.

Next we iterate this process to generate $u^h$ from $u^{h-1}$, $h\geq 2$, yielding estimates of the form \eqref{alas} in $Q_h=\Om\times(0,T_h)$ for some constants $C_h,T_h>0$.
 The difficulty resides in controlling the
constants $C_h$ and $T_h$ independently of $h$. 
In the following lemma, we show how to (re)choose the constants $C_h = \mathfrak C$ and $T_h = \bar T$ uniformly in $h\in \N $ and while keeping $\det du^{h+1}>\e/2$ in 
$Q_h = \Omega \times (0, \bar T)$.
\begin{lemma}\label{lem:choosing good constants}
Using the notation of Proposition \ref{prop:iteration},
if there exist $\mathcal C, \mathcal T>0$ with $\mathcal T\leq 1$ such that
\begin{equation}\label{eqtn:hyp-it}
\|du^h-du_0\|_{C^{\al}(\mathcal  Q)}\le \e \text{ and }\| \p_j \p_l u^h\|_{C^{\al}(\mathcal Q)}\le \mathcal C \|u_0\|_{C^{2,\alpha}(\Om)}
\end{equation}
for $h=1,...,N-1$ and $\mathcal  Q=\Om\times[0,\mathcal T]$, then there exist constants
$\mathfrak C=\mathfrak C(n,p,\|u_0\|_{C^{1,\alpha}(\Om)})>0$ and 
$\mathcal T\ge \mathfrak T=\mathfrak T( \mathcal  C, n,p,\|u_0\|_{C^{1,\alpha}(\Om)})>0$ that are independent of $N$ and such that
$$\|du^N-du_0\|_{C^{\al}(Q)}\le \e \text{ and }\| \p_j \p_l u^N\|_{C^{\al}(Q)}\le \mathfrak  C \|u_0\|_{C^{2,\alpha}(\Om)}$$
in $Q=\Om\times [0,\mathfrak T]$.
\end{lemma}
\begin{proof} 
We set $w^N=u^N-u^{N-1}$ and observe that $w^N$ satisfies 
\begin{equation}
\Bigg\{
\begin{array}{ll}
\p_t w^{N,i} -A_{jl}^{ik}(du^{N-1}) \p_j \p_l w^{N,k}=\bigg[A_{jl}^{ik}(du^{N-1})-A_{jl}^{ik}(du^{N-2}) \bigg]\p_j \p_l u^{N-1,k}  & \text{ in }\mathcal  Q \\
w^{N}=0 &\text{ on }\pa \mathcal  Q.
\end{array}
\end{equation}
Applying the Schauder estimates (\ref{eqn:scaling}) in the cylinder $\Om\times [0,\mathfrak T]$ with $0<\mathfrak T\le
\mathcal T$ to be chosen, we obtain
$$\|dw^N\|_{C^\al(  Q)} \le C(n,p,\e,\|du^{N-1}\|_{C^\al(  Q)}) \mathfrak T^{(1-\al)/2} 
\Bigg\|\bigg[A_{jl}^{ik}(du^{N-1})-A_{jl}^{ik}(du^{N-2}) \bigg]\p_j \p_l u^{N-1,k}\Bigg\|_{C^{\al}( Q)}.$$
The hypothesis \eqref{eqtn:hyp-it} yields a bound on the H\"older norm of the second derivatives
\[
\| \p_j \p_l u^{N-1}\|_{C^{\al}(Q)}\le \mathcal C \|u_0\|_{C^{2,\alpha}(\Om)},
\] 
and at the same time a strictly positive lower bound on $\det du^{h}>\frac{\e}{2}$, for $h=1,\dots,N-1$ in $\mathcal  Q$ so that 
 \begin{multline}
 \|dw^N\|_{C^\al(  Q)} \le C(n,p,\e,\|du^{N-1}\|_{C^\al(Q)}) \mathcal C \|u_0\|_{C^{2,\alpha}(\Om)} \mathfrak T^{(1-\al)/2}  \|dw^{N-1}\|_{C^\al( Q)}
 \\
  \le  C(n,p,\e,\|du_0\|_{C^\al (Q)})  \mathcal C \|u_0\|_{C^{2,\alpha}(\Om)}\mathfrak  T^{(1-\al)/2}  \|dw^{N-1}\|_{C^\al( Q)} .
 \end{multline}
 Choosing $ \mathfrak T$ sufficiently small depending only on $\mathcal C,n,p$ and $\|u_0\|_{C^{2,\alpha}(\Om)}$ 
 and independent of the index $N$, it follows that 
 $$\|dw^N\|_{C^\al(Q)} \le \theta \|dw^{N-1}\|_{C^\al(Q)},$$
 where $\theta\in (0,1)$.
 The latter yields
 $$\|du^N-du_0\|_{C^\al(Q)} \le \sum_{j=1}^N \|du^j- du^{j-1}\|_{C^\al(Q)}
 \le \frac{\theta}{1-\theta} \|du^1-du_0\|_{C^\al(Q)}.$$
We have proved the first part of the conclusion. To establish the estimate
 $\| \p_j \p_l u^N\|_{C^{\al}(Q)}\le \mathfrak C \|u_0\|_{C^{2,\alpha}(\Om)}$ 
 it is now sufficient to apply Schauder estimates to (c) with
$h=N-1$ and observe that the ellipticity bounds on $\Lambda$ and $\lambda$ 
are independent of $N$ in light of the estimate $\|du^N-du_0\|_{C^{\al}(Q)}\le \e$.
\end{proof}
We now complete the proof of Proposition \ref{prop:iteration}.
Applying Lemma \ref{lem:choosing good constants} to the case $h=N=2$ yields bounds
$$\|du^2-du_0\|_{C^{\al}(Q)}\le \e \text{ and }\| \p_j \p_l u^h\|_{C^{\al}(Q)}\le \mathfrak C  \|u_0\|_{C^{2,\alpha}(\Om)}$$
in $Q=\Om\times [0,\mathfrak T]$, with $\mathfrak T=\mathfrak T(C_1 \|u_0\|_{C^{2,\alpha}(\Om)}, n,p,\e)>0$ and $\mathfrak C=\mathfrak C(n,p,\e,  \|u_0\|_{C^{1,\alpha}(\Om)})>0$.
As $\mathfrak C$ is a constant independent of $C_1$, we can eliminate the dependence on $h$ by applying 
Lemma \ref{lem:choosing good constants} again, yielding
\begin{equation}\label{eqn:du^h - du_0 small}
\|du^h-du_0\|_{C^{\al}(Q)}\le \e \text{ and }\| \p_j \p_l u^2\|_{C^{\al}(Q)}\le \mathfrak C \|u_0\|_{C^{2,\alpha}(\Om)}
\end{equation}
for $h=2$ in $Q=\Om\times [0,\bar T]$ with 
\[
\bar T=\bar T(\mathfrak C, n,p,\e,  \|u_0\|_{C^{2,\alpha}(\Om)})=\bar T( n,p,\e,  \|u_0\|_{C^{2,\alpha}(\Om)})>0.
\]
At this point we proceed by induction on $h$: If  (\ref{eqn:du^h - du_0 small}) holds for $h=1,\dots,N$
in $Q=\Om\times [0,\bar T]$ with $\bar T=\bar T( n,p,\e,  \|u_0\|_{C^{2,\alpha}(\Om)})>0$ and $\mathfrak C=\mathfrak C(n,p,\e,  \|u_0\|_{C^{1,\alpha}(\Om)})>0$, then applying Lemma \ref{lem:choosing good constants} at the level of $N+1$ leads to 
(\ref{eqn:du^h - du_0 small}) for $h=N+1$ 
in $Q=\Om\times [0,\bar T]$ with $\bar T$ and $\mathfrak C$ as above; there is no degeneracy of the constants. Finally, since $\mathfrak T$ is
uniform in $h$, (b) follows from the Schauder estimate (\ref{eqn:schauder for our linear system, unscaled}).
This concludes the proof of the proposition.
 \end{proof}

The previous proposition and Arzela-Ascoli theorem yields Proposition \ref{thrm:short time existence non-linear}.
\begin{rmrk}
The proof of the short time existence is quite standard and uses only the Legendre-Hadamard ellipticity rather than the specific structure of the non-linearity in the PDE. It seems plausible to expect that techniques such as those in the paper \cite{Koch-Lamm} may also be used in our setting  to establish short-time existence for  $C^{1,\al}$ initial data.
\end{rmrk}

Note that the Schauder estimates in the appendix yield uniqueness of a $C^{2,\mu}$ solution (for short time) but nevertheless there still may exist rough solutions of the equations with  the same initial data \cite{Muller-Sverak}.

\appendix
%
%
%
%

\section{Existence and basic estimates for classical solutions of parabolic systems}\label{sec:estimates for parabolic systems}
We recall results of Schlag \cite{schlag:1996} and Misawa \cite{misawa:2004} concerning classical 
(i.e.,  two spatial  and one time derivative in $C^\al$) solutions of  the system\footnote{Both papers address more general systems.}
\begin{equation}\label{homogeneous initial data}
\Bigg\{ \begin{array}{ll} \p_t w-A^{ik}_{jl} (x,t) w_{jl}^k = f(x,t) & \text{ in }Q\\ w=0  &\text{ on }\pa Q\end{array}
\end{equation}
assuming that $\Omega$ is a $C^{2,\al}$ domain, $Q=\Om\times(0,T)$, $A,f\in C^{\al}(Q)$,   the compatibility condition $f=0$ on $\p\Omega \times\{t=0\}$ and an ellipticity assumption
\begin{equation}\label{e-schlag}
\lambda |\xi|^2 \le  A_{jl}^{ik}(x,t) \xi^i_j\xi^k_l \le \Lambda |\xi|^2,
\end{equation}
for some $\lambda,\Lambda>0$ and for all $(x,t)\in Q$ and $\xi\in \R^{n\times n}$.

Schlag proves that there exists a constant $C=C(n,\lambda, \Lambda,\|A\|_{0,\alpha}, \Omega)>0$ such that if $w\in C^{2,\al}(Q)$ and $\p_t w \in C^{0,\al}(Q)$ solves 
\eqref{homogeneous initial data} then 
\begin{equation}\label{schauder-schlag}
[w_{jl}]_\al + [w_t]_\al \le C(|w|_0 +[f]_\al).
\end{equation}
Misawa proved that such solutions exist and that the estimate can be slightly strengthened 
\begin{equation}\label{schauder-misawa}
\|w_{jl}\|_{C^\al(Q)} + \|w_t\|_{C^\al(Q)} + \|\nabla w\|_{C^\al(Q)} +\|w\|_{C^\al(Q)} \le C\|f\|_{C^\al(Q)}, \notag
\end{equation}
with a constant $C>0$ depending on $ n,\lambda, \Lambda,\|A\|_{C^{\al}(Q)}, \Omega$ and $T$.

We address the dependence of the constants in the Schauder estimates from the parameter $T$. Since these estimates have a local character we expect
the constant to blow up as $T\to \infty$ and to be bounded for $T>0$  fixed.

Let $T_0\ge T>0$ and set $\frac{1}{\sqrt{T}} \Om:=\{ x\in \R^n | \sqrt{T}x\in \Om\}$. Observe that if
$w$ solves \eqref{homogeneous initial data} then the function
$\tilde w(x,t):= w(\sqrt{T}x,T t)$ solves
\begin{equation}\label{homogeneous initial data rescaled}
\Bigg\{ \begin{array}{ll} \p_t \tilde w-A^{ik}_{jl} (\sqrt{T}x,Tt)\tilde w_{jl}^k = \tilde f(x,t):=Tf(\sqrt{T}x,Tt) & \text{ in }\frac{1}{\sqrt{T}}\Om \times (0,1) \\ \tilde w=0  &\text{ on }\pa \frac{1}{\sqrt{T}}\Om \times (0,1)\end{array} \notag
\end{equation}
Note that $\p_t \tilde w(x,t)= T \p_t w (\sqrt{T}x,Tt)$, $\p_{jl} \tilde w (x,t)= T \p_{jl} w(\sqrt{T}x,Tt)$ and $\nabla \tilde w(x,t)= \sqrt{T} \nabla w (\sqrt{T}x,Tt)$.
The H\"older norm of $\tilde A(x,t):= A(\sqrt{T}x,Tt)$ is bounded by
\[
\min\{1, T^{\alpha/2}\}\|A\|_{C^\al(Q)} \leq \|\tilde A\|_{C^\al(\frac{1}{\sqrt T}\Omega\times(0,1))} \le \|A\|_{C^\al(Q)} (1+T_0^{\al/2}).
\]
Since the ellipticity constants of the coefficients are not affected by the rescaling,
the Schauder estimates for $\tilde w$ read
$$\|\tilde w_{jl}\|_{C^\al(\frac{1}{\sqrt T}\Omega\times(0,1))} + \|\tilde w_t\|_{C^\al(\frac{1}{\sqrt T}\Omega\times(0,1))}
+ \| \nabla \tilde w\|_{C^\al(\frac{1}{\sqrt T}\Omega\times(0,1))}
+\|\tilde w\|_{C^\al(\frac{1}{\sqrt T}\Omega\times(0,1))} \le C\|\tilde f\|_{C^\al(\frac{1}{\sqrt T}\Omega\times(0,1))},
$$
with  a constant $C>0$ depending on $ n,\lambda, \Lambda,\|A\|_{C^{\al}(Q)}, \Omega$. Rescaling back this estimate to the 
parabolic cylinder $Q=\Om\times(0,T)$, we obtain
\begin{equation}\label{eqn:scaling}
\frac{\min\{1,T^{\al/2}\}}{1+T_0^{\al/2}}\Big(\|w_{jl}\|_{C^\al(Q)} + \| w_t\|_{C^\al(Q)} + T^{-1/2} \| \nabla  w\|_{C^\al(Q)} + T^{-1} \| w\|_{C^\al(Q)}\Big)
\le C\| f\|_{C^\al(Q)},
\end{equation}
with $C$ depending on the quantities above and on $T_0$.

Using the standard method based on applying Fourier transform to the integral
$$\int_\Om A^{ik}_{jl} (u^k\phi)_l (u^i\phi)_j\, dx$$
(see for instance \cite{giaquinta:1983}) we note that  the Schauder estimates
continue to hold when weakening the ellipticity assumption from \eqref{e-schlag}
to the Legendre-Hadamard ellipticity
\begin{equation}\label{LH-schlag}
\lambda |\xi|^2|\eta|^2 \le  A_{jl}^{ik}(x,t) \xi^i\xi^k\eta_j\eta_l \le \Lambda |\xi|^2|\eta|^2
\end{equation}
for some $\lambda,\Lambda>0$ and for all $(x,t)\in Q$ and $\xi,\eta\in \R^n$.
Recasting these results for the system 
\begin{equation}\label{eqn:oursystem}
\begin{cases} 
\p_t u^i -  A^{ik}_{jl}(x,t) \p_j \p_l u^k =0  \text{ in } Q \\
u(x,0)=u_0(x)  \text{ for all } x\in\pa Q\end{cases}
\end{equation}
we obtain the following:
\begin{lemma}\label{lemma: existence and regularity for linear systems}
Assume that $\p\Omega$ is $C^{1,\alpha}$, $T_0>0$ and for $0<T\le T_0$,
$Q=\Om\times(0,T)$. If  $A\in C^{0,\al}(Q)$ and  the compatibility condition 
\begin{equation}\label{eqn:compatibility}
A^{ik}_{jl}(x,0) \p_j \p_l u_0^k(x) =0 \textrm{ for all }x\in\p\Omega \textrm{ and }i=1,\dots,n, \notag
\end{equation}
holds, then 
given $u_0\in C^{2,\al}(\Om)$ there exists a solution $u\in C^{2,\al}(Q)$
of \eqref{eqn:oursystem} with $u_t\in C^{0,\al}(Q)$ and such that
\begin{equation}\label{eqn:schauder for our linear system, unscaled}
\|u\|_{C^{2,\al}(Q)} + \|  \p_t u\|_{C^\al(Q)} \le C_1 \|u_0\|_{C^\al(Q)}. 
\end{equation}
The positive constant $C_1$ depends only $T,n,\Omega, \lambda, \Lambda$ and the $C^{2,\al}$ norm of the coefficients of $A$. 
The time-scaled version of \eqref{eqn:schauder for our linear system, unscaled} is
\begin{equation}\label{eqn:schauder for our linear system, scaled}
\frac{\min\{1,T^{\al/2}\}}{1+T_0^{\al/2}}\Big(\|u_{jl}\|_{C^\al(Q)} + \| u_t\|_{C^\al(Q)} + T^{-1/2} \| \nabla  u\|_{C^\al(Q)} + T^{-1} \| u\|_{C^\al(Q)}\Big)
\le C_2\| u_0 \|_{C^\al(Q)},
\end{equation}
where $C_2$ depends only $T_0,n,\Omega, \lambda, \Lambda$ and the $C^{\al}$ norm of the coefficients of $A$. 
\end{lemma}

%
\section{Evolution equations for the Jacobian and the distortion tensor.}
Let $u\in C^1([0,T], C^3(\Omega,\R^n))$ be a classical solution of  \eqref{intro:ivp} and $\tilde\Omega$ be the range of $u_0$ (or equivalently, the
range of $u(\cdot, t)$ for all $t\in[0,T]$).
Denote by $v(\cdot, t)=u^{-1} (\cdot, t)$ the inverse of the diffeomorphism $u$ at time $t$ and set $\beta(y)=\det dv(y).$ For a fixed time $t$, set $y=u(x,t)$ and $dv(y,t)=du^{-1}(x,t)$. Let
$\xi\in C^{\infty}_0(\tilde \Omega\times [0,T], \R)$.

The argument in \cite[Theorem 2.1]{evans-savin-gangbo} yields
\begin{align}
0= &  \int_0^T\int_{\Om} \big[ \p_t u^i-\p_{x_j} A^i_j(du)\big] (\p_{y_i} \xi)\big |_{u} dx\, dt \notag \\ 
= &  \int_0^T\int_{\Om}  \big[ -\p_t (\xi(u(x,t), t)) + \p_t u^i (\p_{y_i} \xi)\big |_{u} \big]  -\p_{x_j} A^i_j(du)(\p_{y_i} \xi)\big |_{u} dx\, dt  
\label{eqn: evolution of determinant}\\
= &  \int_0^T\int_{\Om}  \bigg(-\p_t \xi\big |_u \det dv\big |_u  - \p_{x_j} A^i_j(du)\det dv\big |_u (\p_{y_i} \xi)\big |_{u} \bigg) \ \det du \ dx\, dt \notag 
\end{align}

Next, we define $$\tilde A^i_j(q)= A^i_j(q^{-1})=-p\bigg( q_{ji}- n \frac{q^{ij}}{|q^{-1}|^2}\bigg)\frac{|q^{-1}|^{np}}{(\det q)^p}$$ for all non-singular matrices $q$
and observe that
$$\p_{x_j} A^i_j(du(x,t))= \p_{x_j} \tilde A^i_j( dv(u(x,t), t) )=dv^{hj}(u(x,t),t)[ \p_{y_h} \tilde A^i_j( dv(y, t) ) ]\big|_u.$$
Also, on $\tilde\Omega$,
\begin{align*}
dv^{hj}  \p_{y_h} \tilde A^i_j( dv)  \beta 
&= -\p_{y_h}\bigg[p(\cof dv)_{jh}\bigg( du^{ji}-n \frac{du_{ij}}{|du|^2}\bigg)\bigg|_v \frac{|du|^{np}}{(\det du)^p}\bigg|_v\bigg] \\
&=-\p_{y_h} \bigg[p(\cof dv)_{jh} \bigg( du^{ji}-n \frac{du_{ij}}{|du|^2}\bigg)\bigg|_v \frac{|du|^{np}}{(\det du)^p}\bigg|_v\bigg] \\
&=- \p_{y_h} \bigg[ p \beta \bigg( \delta_{hi} - n \frac{ du_{hj} du_{ij}}{|du|^2} \bigg)  \bigg|_v
   \frac{|du|^{np}}{(\det du)^p}\bigg|_v\bigg]
\end{align*}
since $\p_{y_h} (\cof dv)_{jh}=0$.
The latter and \eqref{eqn: evolution of determinant} yield

\begin{align*} 
0=& 
 \int_0^T\int_{\Om}  \bigg(-\p_t \xi\big |_u \det dv\big |_u  -dv^{hj}\big |_u[ \p_{y_h} \tilde A^i_j( dv(y, t) ) ]\big|_u \det dv\big |_u (\p_{y_i} \xi)|_{u} \bigg) 
 \det du \ dx\, dt  \\
=& \int_0^T\int_{\tilde \Om} -\p_t \xi \beta - \bigg[dv^{hj}[ \p_{y_h} \tilde A^i_j( dv(y, t) ) ] \beta \bigg] \p_{y_i} \xi dy\, dt  \\
 =& \int_0^T\int_{\tilde \Om} \xi\p_t \beta -\p_{y_i}\p_{y_h} \bigg[
 \beta \bigg( \delta_{hi} - n \frac{ du_{hj} du_{ij}}{|du|^2} \bigg)  \Bigg|_v
 \frac{|du|^{np}}{(\det du)^p}\Bigg|_v\bigg]  \xi \ dy\, dt.
\end{align*}

We have then proved the following:

\begin{lemma} \label{lem:determinant pde}
Let $u\in C^1([0,T], C^3(\Omega,\R^n))$ be a classical solution of  \eqref{intro:ivp}.
If we set $v(\cdot, t)=u^{-1} (\cdot, t)$ and  $\beta(y)=\det dv(y)$, then  $\beta$ satisfies

\begin{equation}\label{eqn: determinant pde}
 \p_t \beta  = \p_{y_i} \p_{y_h} \bigg[B_{ih}(du)\bigg|_v \beta \bigg],
 \end{equation}
 in $\tilde \Omega \times (0,T)$, with $$B_{ih}(du)= p \bigg( \delta_{hi} - n \frac{ du_{hj} du_{ij}}{|du|^2} \bigg) 
   \frac{|du|^{np}}{(\det du)^p},$$ as well as Neumann type conditions
 $$\p_\nu  \Big[dv^{jh}[ \p_{y_h} \tilde A^i_j( dv ) ] \beta \Big]= \p_\nu\p_{y_h} \Big[B_{ih}(du)\Big|_v \beta \Big] = 0$$
 for all $(y,t)\in \p \tilde \Om \times (0,T).$
\end{lemma}

Let $\eta\in \R^n$ and $q$ a non-singular matrix, and consider the quantity
$$B_{ih}(q)\eta^i\eta^h= p \bigg( \delta_{hi} \eta^i \eta^h- n \frac{ [(\eta q)^j]^2}{|q|^2} \bigg) 
   \frac{|q|^{np}}{(\det q)^p}=p|\eta|^2 \bigg( 1- \frac{ \frac{|\eta q|^2}{|\eta|^2}}{\frac{|q|^2}{n}}\bigg) \frac{|q|^{np}}{(\det q)^p}$$

In the model case when $n=2$, $q$ is diagonal with  eigenvalues $0<\lambda_1\le \lambda_2$,  and for unit $\eta$, one has
$$B_{ih}(q)\eta^i\eta^h=p\bigg( 1- \frac{\eta_1^2 \lambda_1^2+\eta_2^2\lambda_2^2}{\frac{\lambda_1^2+\lambda_2^2}{2}}\bigg) \bigg(\frac{\lambda_1^2+\lambda_2^2}{\lambda_1\lambda_2}\bigg)^p$$

The matrix does not have a sign, it vanishes when $\lambda_1=\lambda_2$.
Unlike the case studied in \cite{evans-savin-gangbo},  the parabolic maximum principle does not apply.

\begin{lemma} If $u$ is as in Lemma \ref{lem:determinant pde}
 then the corresponding conformal metric evolves according to
\begin{multline}
\p_t g_{\al \beta} =n p\fbeta^{2/n} \Bigg\{ \p_k\p_j \bigg[ S(g)_{l \al} du^{-1}_{jl} K^{np-2}\bigg] du_{\beta k} 
+  \p_k\p_j \bigg[ S(g)_{l\beta} du^{-1}_{jl} K^{np-2}\bigg] du_{\alpha k}
\\
- \frac{2}{n}\fbeta^{-1}\bigg[ \p_{x_j} \p_{x_k} \Bigg(  S_{ih}(g) \det du^{-1}K^{np-2} \bigg) du^{kh} du^{ji} -\p_{x_k}  \Bigg(  S_{ih}(g) \det du^{-1} K^{np-2} \bigg) 
du^{sh} du^{kl} \p_{x_j}\p_{x_s} u^l du^{ji} \\+(\det du)^{-1}du^{k l}du^{is}   \p_k \p_s u^l
 \p_j \bigg( S(g)_{m  i} du^{-1}_{jm} K^{np-2}\bigg) \bigg] du_{\al h_1} du_{\beta h_1} \Bigg\}
\end{multline}
in $Q=\Om\times (0,T)$ with $g=g_0$ on $\pa Q$, and where $ \fbeta=(\det du)^{-1}=\det du^{-1} \circ u$.
\end{lemma}

\bibliographystyle{acm}
\bibliography{papers}
\end{document}